\pdfoutput=1
\documentclass[11pt]{amsart}
\usepackage{amsmath,amssymb,amsthm}
\usepackage{mathtools}
\usepackage[margin=1.2in]{geometry}
\usepackage{booktabs}
\usepackage{array}
\usepackage{flafter}
\usepackage{xcolor}
\usepackage[colorlinks=true,linkcolor=blue,citecolor=blue,urlcolor=blue]{hyperref}
\hypersetup{pdfauthor={Bernd Johannes Wuebben},
  pdftitle={A vanishing-cycle obstruction to smoothing Calabi-Yau threefolds},
  pdfsubject={Global smoothing obstructions for singular Calabi-Yau threefolds},
  pdfkeywords={Calabi-Yau threefolds, smoothing of singularities, vanishing cycles,
  del Pezzo cones, ordinary double points, Milnor fibers, reflexive polytopes}}

\newtheorem{theorem}{Theorem}[section]
\newtheorem{proposition}[theorem]{Proposition}
\newtheorem{lemma}[theorem]{Lemma}
\newtheorem{corollary}[theorem]{Corollary}

\theoremstyle{definition}
\newtheorem{definition}[theorem]{Definition}
\newtheorem{example}[theorem]{Example}

\newtheorem{convention}[theorem]{Convention}
\theoremstyle{remark}
\newtheorem{remark}[theorem]{Remark}

\theoremstyle{plain}
\newtheorem{thmletter}{Theorem}

\newcommand{\ZZ}{\mathbb{Z}}
\newcommand{\QQ}{\mathbb{Q}}
\newcommand{\RR}{\mathbb{R}}
\newcommand{\CC}{\mathbb{C}}
\newcommand{\PP}{\mathbb{P}}
\newcommand{\cO}{\mathcal{O}}
\newcommand{\Rsm}{\mathcal{R}}   
\newcommand{\Mil}{\mathcal{M}}  
\newcommand{\cX}{\mathcal{X}}
\newcommand{\conv}{\operatorname{conv}}
\newcommand{\Cone}{\operatorname{Cone}}
\newcommand{\Sing}{\operatorname{Sing}}
\newcommand{\Pic}{\operatorname{Pic}}
\newcommand{\Proj}{\operatorname{Proj}}
\newcommand{\AAf}{\mathbb{A}}
\newcommand{\TV}{\operatorname{TV}}
\newcommand{\Cl}{\operatorname{Cl}}
\newcommand{\Def}{\operatorname{Def}}
\newcommand{\rk}{\operatorname{rank}}
\newcommand{\Bl}{\operatorname{Bl}}
\newcommand{\dP}{\mathrm{dP}}
\newcommand{\GL}{\mathrm{GL}}
\newcommand{\Xh}{\widehat{X}}
\newcommand{\Sb}{\bar{S}}
\newcommand{\Disk}{\mathbb{D}}

\title[A vanishing-cycle smoothing obstruction]{A vanishing-cycle
obstruction to smoothing\\ Calabi--Yau threefolds}
\author{Bernd Johannes Wuebben}

\date{27 August 2026}

\begin{document}

\begin{abstract}
We prove a homological obstruction to smoothing projective Calabi--Yau
threefolds with nodes and anticanonical cones over smooth del Pezzo surfaces
of degrees six and seven. On a crepant resolution that is small over the
nodes and blows up the cone vertices, any smoothing forces a relation among
link classes bounding in the local Milnor fibers. The relation is nonzero
at every node and degree-seven point, and at each degree-six point smoothed
through the one-dimensional component. The proof uses vanishing periods
and applies to smoothing arcs of arbitrary contact order, extending the
necessity direction of Friedman's criterion. We apply the obstruction to a
Batyrev hypersurface whose thirty singularities are individually smoothable
but which admits no smoothing, even formally. On its mirror partner, two
nodes cannot be smoothed by any analytic deformation, independently of a
locally nonsmoothable cone point.
\end{abstract}

\maketitle

\section{Introduction}\label{sec:intro}

\subsection{The problem}
Let $X$ be a projective threefold over $\CC$ with trivial dualizing
sheaf and isolated singularities.  When every singular germ of $X$
admits a smoothing, one may ask whether $X$ itself does.  For ordinary
double points the answer is classical: by Friedman's
criterion~\cite{Friedman86}, simultaneous smoothing is governed by a
relation among the exceptional curves of a small resolution in which
every curve appears with nonzero coefficient; see also
Rollenske--Thomas~\cite{RT09} for a topological formulation in all odd
dimensions.  Outside the nodal case much less is known.  The local
deformation theory of the simplest non-nodal germs, the
anticanonical cones over del Pezzo surfaces, has been understood
since Altmann~\cite{Altmann95,Altmann97}, and their smoothing
components have explicit geometric descriptions
(Gross~\cite[\S5.5]{Gross97}; Petracci~\cite{Petracci20,Petracci19}).
Globally, Namikawa proved that a $\QQ$-factorial Calabi--Yau threefold with
isolated toric Gorenstein singularities is smoothable whenever its germs are,
and that for a contracted del Pezzo surface of degree~$6$ the restriction of
divisors to the surface determines which smoothing components of the germ are
reached by global deformations \cite{Namikawa97}.  We prove a necessary
condition for actual smoothings that couples nodal and divisorial link
classes on threefolds with nodes and del Pezzo cone points, without
assuming $\QQ$-factoriality or unobstructedness.

The first three papers of this program~\cite{PaperI,PaperII,PaperIII}
study the isolated Gorenstein toric threefold germs that arise on Batyrev
hypersurfaces from unit-edge polygons.  They use a local Minkowski-sum
obstruction to smoothing and classify the germs with small edge
vectors~\cite{PaperI}, and, for faces whose germs are smoothable and when a
dual-edge-length invariant is at least two, construct deformations that
smooth the corresponding points, which give global
smoothings under further hypotheses~\cite{PaperII}.  They also identify a
Batyrev mirror pair, associated with a $22$-vertex reflexive $4$-polytope
$\Delta_{\mathbb F_1}$ and its polar, for which both members have only isolated
singularities.  The remaining smoothing question is \emph{mixed}: nodes and
del Pezzo cone points occur on the same threefold, so any smoothing must move
both kinds at once.  One member carries a cone singularity with no smoothing
and is locally obstructed.  The other, denoted $X^{\circ}$ throughout, has
$26$ ordinary double points, two
anticanonical cones over the del Pezzo surface of degree~$7$, and two
over the del Pezzo surface of degree~$6$: thirty germs, every one of
them locally smoothable.  A purely nodal analysis cannot decide
$X^{\circ}$: two of its twenty-six nodes are \emph{coloops} of the relation
matrix, meaning that their coordinate vanishes on every relation among the
classes of the exceptional curves of a small resolution, so no nodal relation is
nonzero there and any smoothing relation would have to include classes
from the exceptional divisors.

\subsection{Results}
Fix the following setting.  $X$ is a connected compact complex threefold with
$\omega_X\cong\cO_X$ and $h^1(\cO_X)=0$ whose singular points are
ordinary double points and anticanonical cones over del Pezzo surfaces
of degrees~$6$ and~$7$; $\pi\colon\Xh\to X$ is a crepant resolution,
small over the nodes, and near each cone point $p$ isomorphic to the
blowup of the vertex, so that $\pi^{-1}(p)$ is a single exceptional
surface $E_p$, a copy of the del Pezzo surface of the cone.  For a singular point $p$ let $L_p$ denote
its link.  A \emph{one-parameter deformation} of $X$ is a flat proper
analytic family $\cX\to\Disk$ over a disk $\Disk$ with
$\cX_0\cong X$, and it is a \emph{one-parameter smoothing} if $X_s$ is
smooth for all $s\neq0$.  Such a smoothing induces at each singular
point a smoothing of the local germ, with Milnor fiber $\Mil_p$, and we
set
\[
\Rsm_p\ :=\ \ker\bigl(H_2(L_p;\QQ)\longrightarrow H_2(\Mil_p;\QQ)\bigr),
\]
the part of the link homology that bounds in the local Milnor fiber.  By
Altmann's computation, recalled in Section~\ref{sec:miniversal}, the miniversal
base of a node or of a degree-$7$ cone germ has one reduced component, and that
of a degree-$6$ cone germ has two, of dimensions $1$ and $2$; a smoothing sends
the germ into one of them, and $\Rsm_p$ depends on which.

\begin{thmletter}[mixed necessity; $=$ Theorem~\ref{thm:necessity}]\label{thm:necessity-intro}
Suppose $X$ as above admits a one-parameter smoothing.  Then the space
\[
K(X)\ :=\ \Bigl\{\kappa\in\ker\Bigl(\bigoplus_pH_2(L_p;\QQ)
\to H_2(\Xh;\QQ)\Bigr)\ :\ \kappa_p\in \Rsm_p
\text{ at every cone point }p\Bigr\}
\]
contains, for every node $q$, for every degree-$7$ cone point $q$, and
for every degree-$6$ cone point $q$ whose induced local smoothing lies
in the one-dimensional component of its miniversal base, a class
$\kappa$ with $\kappa_q\neq0$.  Here the maps are induced by the
inclusions $L_p\subset X\smallsetminus\Sing X\subset\Xh$.
\end{thmletter}

The subspaces $\Rsm_p$ are computed in closed form
(Proposition~\ref{prop:kernels}): at a degree-$7$ point $\Rsm_p$ is the
unique line spanned by a $(-2)$-class of
$K^{\perp}\subset H_2(E_p;\QQ)$, and at a degree-$6$ point it is the
$A_1$ or the $A_2$ summand of the root lattice
$K^{\perp}\cong A_1\oplus A_2$, according to the component of the local
miniversal base through which the family smooths the germ.  These
subspaces coincide with the lattices $\Xi(V)$ attached by
Kuznetsov--Prokhorov~\cite[Theorem~1.9]{KP23} to the del Pezzo
threefolds $V$ that sweep out the local smoothings.

Theorem~\ref{thm:necessity-intro} is a finite linear-algebra test on any
concrete $X$.  It is finite although $\Rsm_p$ refers to a smoothing whose
existence is in question, because the degree-$6$ points contribute only finitely
many choices of component: one runs the test on each, as
Section~\ref{sec:thetest} does.  On the distinguished hypersurface every choice
obstructs.

\begin{thmletter}[$=$ Theorem~\ref{thm:main} and Corollary~\ref{cor:def}]\label{thm:xcirc-intro}
Let $\Delta_{\mathbb F_1}$ be the $22$-vertex reflexive polytope printed in
Table~\ref{tab:partnerrays}, with polar $\Delta_{\mathbb F_1}^{\vee}$ printed
in Table~\ref{tab:rays}.  Let $X^{\circ}$ be an anticanonical hypersurface,
regular in Batyrev's sense \cite[Def.~3.1.1]{Batyrev94}, in the Gorenstein
Fano toric fourfold whose fan is the face fan of
$\Delta_{\mathbb F_1}^{\vee}$, equivalently whose anticanonical polytope is
$\Delta_{\mathbb F_1}$.  Then $X^{\circ}$ admits no one-parameter
smoothing.  Moreover no fiber of a small representative of the
miniversal deformation of $X^{\circ}$ is smooth, and no formal
arc in the miniversal deformation of $X^{\circ}$ meets all thirty
local smoothing loci, in the sense of Corollary~\ref{cor:def}.
\end{thmletter}

Theorems~\ref{thm:necessity-intro} and~\ref{thm:xcirc-intro} are the main
results.  The first is the general topological necessity statement; the second
is its application to the mixed Batyrev hypersurface that motivated the
criterion.

Call a Batyrev mirror pair \emph{doubly isolated} when both members have only
isolated singularities.  The mirror partner
$X_{\Delta}$ carries an anticanonical cone over $\mathbb F_1$, a germ with no
smoothing~\cite{PaperI,PaperIII}, so both members of the pair are
non-smoothable.  The criterion says more.  Extended to deformations
that leave such a germ in place (Section~\ref{sec:partner}), it shows
that $X_{\Delta}$ carries the global obstruction as well: two of its
twenty nodes are smoothed by no one-parameter deformation at all, and
this is
independent of the $\mathbb F_1$ germ, which already prevents a
smoothing for
a purely local reason.  The two members of the pair are therefore not
an instance of one obstruction seen locally and globally; each fails
globally, and one of them fails locally in addition.

The analytic core of Theorem~\ref{thm:necessity-intro} is quantitative: the
period of the holomorphic $3$-form of the nearby fiber over the
vanishing cycle of a node equals the local smoothing parameter times a
unit (Lemma~\ref{lem:P}; for the linearity of the nodal period see
also \cite[(1.4)]{RT09}), and an analogous nonvanishing holds for the
distinguished cycles of the degree-$7$ smoothings and the degree-$6$
smoothings through the one-dimensional component (Lemma~\ref{lem:D}), where
the decisive constant is an
explicit period equal to $4\pi^2$ on a cycle we construct by hand.
Consequently a smoothed germ cannot have rationally null-homologous
vanishing homology, at any order of tangency of the family with the
local discriminants.  In the purely nodal case this yields:

\begin{corollary}[Corollary~\ref{cor:friedman}]\label{cor:friedman-intro}
For a connected compact complex threefold with $\omega\cong\cO$,
$h^1(\cO)=0$ and only ordinary double points, any one-parameter
smoothing forces, for every node $q$, the existence of a relation
among the exceptional curve classes of a small resolution that is
nonzero at $q$.  This is the necessity half of Friedman's criterion.  It is not new in the
nodal case: \cite[Theorem~2.5]{Namikawa02} proves the equivalence for
threefolds with terminal singularities, and \cite[\S4]{Collins25} gives a
differential-geometric proof of necessity.  What the method adds is that it
needs no unobstructedness and no hypothesis of a smooth total space, and
that it survives at germs which are not local complete intersections, where
those two literatures stop.
\end{corollary}

For $X^{\circ}$ the exact finite input localizes the obstruction
topologically:

\begin{proposition}[Proposition~\ref{prop:coloops}]
\label{prop:coloops-intro}
In any hypothetical one-parameter smoothing of $X^{\circ}$, the
vanishing $3$-spheres of the two coloop nodes would be null-homologous
in $H_3(X_s;\QQ)$.
\end{proposition}

Lemma~\ref{lem:P} says that a smoothed node has a nonzero vanishing-cycle
class, whereas Proposition~\ref{prop:coloops-intro} makes that class zero at
either coloop node.  This contradiction proves
Theorem~\ref{thm:xcirc-intro}.  A
complementary second proof, with a different period input, runs
through a degree-$7$ point instead of
the nodes (Section~\ref{sec:proofs}).

The one-germ result and the bounded census are applications of the general
criterion rather than additional main theorems.  On the seven-vertex polytope
$\Delta_7$, the unique singularity is a degree-$7$ cone and its root line
survives in $H_2(\Xh;\QQ)$, so Corollary~\ref{cor:onegerm} obstructs a
smoothing.  Among the $77$ admissible polytopes with at most nine vertices and
a pentagonal $2$-face, the same test obstructs $76$
(Proposition~\ref{prop:census}).

Section~\ref{sec:measures} interprets the matrix used in the finite test.
Choose a basis of the link homology at each germ.  Each basis class defines
a linear functional on the ambient Picard group by intersection with its
image on the resolution; its values on ambient divisor classes form a
\emph{row} of the matrix $B$.  Read in local ray coordinates, the rows at a
germ form a rational basis whose complexification is $T^1$
(Proposition~\ref{prop:T1channels}).  For the isolated hypersurfaces under
consideration, vanishing of Batyrev's correction term implies that the
ambient divisors detect the entire topological relation kernel
(Proposition~\ref{prop:kernelequal}).  Finally, under the resolution
hypotheses of Proposition~\ref{prop:qfact}, algebraic $\QQ$-factoriality of
$X$ forces every divisorial row to vanish.  This explains compatibility
with Gross's smoothing theorem for primitive type~II contractions
\cite[Theorem~5.8]{Gross97}.

The criterion does not obstruct every mixed example.  The anticanonical
hypersurface $X_{19}$ of a $19$-vertex reflexive polytope from the
companion paper~\cite{PaperIII} has fourteen nodes and one degree-$7$
cone point, and Section~\ref{sec:X19} shows that it again has two nodal
coloops.  Applied to $X_{19}$, the finite test is
silent: no coordinate to which Theorem~\ref{thm:necessity-intro}'s
nonvanishing conclusion applies is forced to vanish
(Section~\ref{sec:test}).

\begin{example}\label{q:X19}
The companion paper~\cite{PaperV} proves that $X_{19}$ is smoothable and
that its full deformation space is smooth of dimension $30$.  Every
smoothing must deform its degree-$7$ cone point as well as its nodes;
the nodal relation space alone does not satisfy the necessity condition.
\end{example}

\subsection{Method}
The general necessity theorem has two inputs.  First, the vanishing-period computations of
Section~\ref{sec:periods} convert nonvanishing of local smoothing parameters
into nonvanishing homology classes.  Second, the local-to-global membership
statement requires a boundary-marked topological comparison.  The germ Milnor
fiber is identified with a del Pezzo threefold complement as a manifold with
boundary, compatibly with the two natural circle fibrations of the link
(Lemma~\ref{lem:boundary}); together with the classification of del Pezzo
threefolds of degrees~$6$ and~$7$ (Fujita~\cite{Fujita80a}), this computes the
subspaces $\Rsm_p$.  Surgery along the links and Mayer--Vietoris then place the
boundary of every $3$-cycle of $X_s$ in the space $K(X)$ of
Theorem~\ref{thm:necessity-intro}.  Its application to $X^{\circ}$ adds a third
input: an exact
integral computation on a fixed crepant resolution shows that on the
relevant restricted kernels the coordinates of the two coloop nodes
and of the degree-$7$ directions vanish identically
(Section~\ref{sec:example}); all decisive matrix data are printed, and
Appendix~\ref{app:cert} lists the programs that recompute them.

The argument does not use the first-order deformation theory of $X$:
the localization map and the comparison of coordinates on local
deformation spaces do not enter.  The periods apply directly to the
smoothing arc, including arbitrary orders of contact with the local
discriminants.  For nodes, the existence criterion for an actual
smoothing is already stated in \cite[Thm.~1.3]{RT09}.  The point here is
to obtain the mixed necessity statement directly from periods, without
assuming unobstructedness of the deformation space of $X$.

\subsection{Relation to prior work}
Friedman~\cite{Friedman86} established the relation criterion for
first-order smoothings of nodal Calabi--Yau threefolds.
Rollenske--Thomas~\cite[Thms.~1.2--1.3]{RT09} recall both that statement
and the criterion for actual smoothings obtained using unobstructedness,
and develop its higher-dimensional counterpart.  Namikawa
\cite[Theorem~2.5]{Namikawa02} proved the nodal statement as an
equivalence, for terminal singularities, and Collins
\cite[\S4]{Collins25} proves necessity there by a differential-geometric
argument.  In the opposite direction, Aguilar \cite{Aguilar26} shows that
for log conifold transitions of Fano threefold pairs of index two the
analogous balancing conditions vanish and
local first-order smoothings of nodes lift independently, so no relation is forced;
the contrast marks how much of the present obstruction is Calabi--Yau
specific.  Ning~\cite{Ning26} extends local-freeness results beyond the local
complete intersection case.  Friedman's generalized $T^1$ lifting
theorem likewise permits isolated Du Bois singularities that are not
local complete intersections \cite[Thm.~3.10]{Friedman25}.
The companion paper~\cite{PaperV} combines that theorem with crepant
partial models and analytic blow-down to prove integrability of the
selected deformation spaces under its stated projection hypothesis.
Friedman--Laza
\cite[\S5]{FL22a} identify the first-order localization data of
Calabi--Yau varieties with isolated local complete intersection
singularities in terms of link homology, and \cite{FL22b} develops the
corresponding local theory; our
Section~\ref{sec:example} uses only the resolution-side avatar of
that picture, in the form of an integral matrix.
Altmann~\cite{Altmann95,Altmann97} computed the miniversal deformations of
the del Pezzo cone germs and their smoothing components, which
Gross~\cite[\S5.5]{Gross97} describes geometrically;
Petracci~\cite{Petracci20,Petracci19} and
Kaloghiros--Petracci~\cite{KP21} give modern treatments, the latter
computing the Betti numbers of the degree-$6$ Milnor fibers, in
agreement with Section~\ref{sec:germs}.  The classification of the
sweeping del Pezzo threefolds is due to Fujita~\cite{Fujita80a}; we
follow Fujita in recording that the degree-$6$ case of Iskovskikh's earlier list
omits $\PP(T_{\PP^2})$, and cite Fujita, with
Kuznetsov--Prokhorov~\cite{KP23} as the modern reference.  On the
constructive side, multi-parameter deformations of toric varieties in
a single character degree were built by
Ilten--Vollmert~\cite{IV12}.  The companion paper~\cite{PaperV}
proves the mixed smoothing criterion, including the two-dimensional
degree-six component, and proves the smoothability of
$X_{19}$ in Example~\ref{q:X19}.

\begin{convention}\label{conv:dp}
$\dP_d$ denotes a del Pezzo \emph{surface} of degree $d=K^2$; thus
$\dP_6=\Bl_3\PP^2$ and $\dP_7=\Bl_2\PP^2$.  A \emph{del Pezzo
threefold} is a smooth projective threefold $V$ with $-K_V=2H$ for an
ample $H$; its degree is $H^3$.  For a polygon or polytope, edges are
\emph{unit} if they have lattice length one.  Homology and cohomology
have $\QQ$-coefficients unless stated otherwise.  The anticanonical
cone over a del Pezzo surface $S$ is
$C(S)=\operatorname{Spec}\bigoplus_{k\ge0}H^0(S,-kK_S)$.
\end{convention}

\subsection*{Acknowledgements}
I am grateful to Mark Gross for first drawing my attention, many years ago, to the problem of smoothing canonical singularities on Calabi--Yau threefolds.
The author used LLMs during the development of this work for literature exploration, editorial and coding assistance and as an adversarial reader.

\section{The germs: links, miniversal bases, Milnor fibers}
\label{sec:germs}

Throughout this section $(Y,0)=C(S)$ is the anticanonical cone over
a smooth del Pezzo surface $S$, an isolated Gorenstein toric threefold
singularity when $S$ is toric.  For $S=\dP_6$, $\dP_7$ and
$\mathbb F_1$, the cases used in this paper, it is the toric variety
of the cone over the unit-edge reflexive hexagon, pentagon and
quadrilateral respectively, placed at height one.  (The
unit-edge polygon here is the \emph{fan} polygon of $S$; its polar,
the Newton polygon of $-K_S$, is not unit-edge for the pentagon.)

\subsection{Links}\label{sec:links}
The punctured cone is the total space of $\cO_S(K_S)$ minus its zero
section, so the link $L$ is the unit circle bundle of $\cO_S(K_S)$,
and the homology Gysin sequence together with $H_1(S)=0$ gives a
canonical isomorphism
\begin{equation}\label{eq:linkH2}
H_2(L;\QQ)\ \cong\ K_S^{\perp}\ \subset\ H_2(S;\QQ),
\end{equation}
of rank $3$ for $\dP_6$ and $2$ for $\dP_7$.  For a node,
$L\cong S^2\times S^3$ and $H_2(L;\QQ)=\QQ$.  The lattice
$K^{\perp}(\dP_6)$ is the root lattice $A_1\oplus A_2$ (discriminant
$6$); $K^{\perp}(\dP_7)$ has discriminant $7$ and contains exactly one
pair $\pm$ of primitive classes of square $-2$.

\subsection{Miniversal deformations}\label{sec:miniversal}
By Altmann's theorems, the deformation theory of these germs is
completely explicit.  $T^1$ is concentrated in the Gorenstein
character degree, of dimension $n-3$ for an $n$-gon
\cite[(6.5) and the subsequent Remark]{Altmann95}; the versal family
in that degree, constructed in \cite{Altmann97}, is therefore a versal
deformation of the affine cone, with base ring computed there:
\[
\begin{aligned}
\dP_7&:\ \ \CC[s_1,s_2]/(s_1^2,\,s_1s_2)
&&\text{\cite[(9.1)(v)]{Altmann97}},\\
\dP_6&:\ \ \CC[s_1,s_2,s_3]/(s_1s_3,\,s_2s_3)
&&\text{\cite[(2.5)]{Altmann97}}.
\end{aligned}
\]
The reduced components correspond to the maximal decompositions of the
polygon into Minkowski sums of lattice polygons
\cite[(2.5),\,(8.1)]{Altmann97}: the pentagon decomposes in exactly
one way, into a unit segment (its pair of parallel edges) and a
unimodular triangle, giving
one reduced component, the line $\{s_1=0\}$; the hexagon decomposes as
a sum of two triangles (the line $\{s_1=s_2=0\}$) and as a sum of
three segments (the plane $\{s_3=0\}$).

We record three standard supplements used later.
First, \emph{minimality}: the Kodaira--Spencer map of Altmann's family
is an isomorphism onto $T^1$ \cite[(6.2),\,(6.3)]{Altmann97}, so the
versal family is miniversal.  Second, \emph{the germ}: $T^1$ and $T^2$
of the affine cone are finite-dimensional and supported at the vertex,
and the deformation theory of the affine variety is identified with
that of the analytic germ $(Y,0)$ by Artin
approximation~\cite{Artin69}; the analytic miniversal deformation
exists by Grauert~\cite{Grauert72}; the base is the analytic germ of
the ring displayed above.  (The versality criterion used in
\cite[(7.2)]{Altmann97} is that of de~Jong--van~Straten
\cite[\S4]{dJvS94}.)  Third, \emph{arcs}: a morphism from the germ of
a disk to the miniversal base factors through the reduction and,
since the analytic local ring of a disk is a domain, lies in a single
reduced component; for the pentagon base this forces $s_1\equiv0$
along any arc, and for the hexagon base it forces the arc into the
line or into the plane.  We refer to the corresponding local
smoothings as smoothings \emph{through} the given component.

\subsection{Sweeps and Milnor fibers}\label{sec:sweeps}
The smoothing components of these cones are cones swept by hyperplane
sections, in the sense of Pinkham
\cite[(4.2),\,(5.1),\,(5.5)]{Pinkham74}; see
\cite[Lemma~4.1, Proposition~4.2]{PP15} for the statement in all
dimensions in the pencil form used here.  Concretely
(Gross~\cite[\S5.5]{Gross97};
Petracci~\cite[Proposition~2.2]{Petracci20},
\cite[Proposition~4.12]{Petracci19}): the general fiber of the
degree-$7$ smoothing is $V_7\smallsetminus\Sb$ with
$V_7=\Bl_{pt}\PP^3$; the general fiber of the hexagon's one-dimensional
(line) component is $(\PP^1)^3\smallsetminus\Sb$; and that of its
two-dimensional (plane) component is
$W_6\smallsetminus\Sb$ with $W_6=\PP(T_{\PP^2})$ the divisor of bidegree
$(1,1)$ in $\PP^2\times\PP^2$.  In each case $V$ is a del Pezzo
threefold of the same degree with $\Sb\in|H|$ a copy of $S$, embedded
by $-K_S=H|_{\Sb}$; the classification of the possible $V$
(Fujita~\cite[(5.8),\,(5.16)]{Fujita80a}) shows there are no others.
For the record, the open Milnor fibers of the three branches have
$(b_2,b_3)=(1,1)$, $(2,1)$ and $(1,2)$ respectively; the degree-$6$
values agree with \cite[Proposition~3.5]{KP21}.

\begin{proposition}\label{prop:kernels}
Let $\Mil=V\smallsetminus\Sb$ be one of the three sweep complements.
Under the identification~\eqref{eq:linkH2} of $H_2$ of the link
$L=\partial(\text{tube around }\Sb)$ with $K^{\perp}_{S}$,
\[
\ker\bigl(H_2(L;\QQ)\to H_2(\Mil;\QQ)\bigr)
=\ker\bigl(K^{\perp}_S\to H_2(V;\QQ)\bigr)=
\begin{cases}
\text{the }(-2)\text{-line} & V=V_7,\\
A_1 & V=(\PP^1)^3,\\
A_2 & V=W_6,
\end{cases}
\]
where the map is the restriction to $K^\perp$ of the pushforward
$H_2(S;\QQ)\to H_2(V;\QQ)$.  For a node,
$\ker(H_2(L;\QQ)\to H_2(\Mil;\QQ))=H_2(L;\QQ)=\QQ$.
\end{proposition}

\begin{proof}
Since $H_1(\Sb)=0$, the Thom
isomorphism gives $H_3(V,\Mil)\cong H_1(\Sb)=0$, so
$H_2(\Mil)\hookrightarrow H_2(V)$, and the inclusion $L\subset \Mil$
composed with this injection factors, up to homotopy through the
tube, as the projection $L\to\Sb$ followed by $\Sb\subset V$; the two
kernels therefore coincide.  (Here $L$ is the circle bundle of the
normal bundle $\cO_{\Sb}(\Sb)\cong\cO_{\Sb}(-K_{\Sb})$; its Gysin
sequence identifies $H_2(L;\QQ)$ with the same subspace $K^{\perp}$ as
in~\eqref{eq:linkH2}.)  The pushforward is computed in the
standard bases.  For $V_7$: a smooth member of $|H|=|2H_{\PP^3}-E|$
is the strict transform of a quadric through the blown-up point, so
$\Sb\cong\Bl_2\PP^2$ with $H_{\PP^3}|_{\Sb}=2\ell-e_1-e_2$ and
$E|_{\Sb}=\ell-e_1-e_2$; the pushforward matrix has rows
$(2,1),(1,1),(1,1)$ on $(\ell,e_1,e_2)$ and kernel
$\langle e_1-e_2\rangle$, which is the $(-2)$-line.  For
$(\PP^1)^3$: the three projections restrict to the conic pencils
$|\ell-e_i|$, the rows are $(1,1,1)$ and the unit vectors, and the
kernel is $\langle\ell-e_1-e_2-e_3\rangle=A_1$.  For $W_6$: the two
contractions restrict to $\ell$ and $2\ell-e_1-e_2-e_3$, and the
kernel is $\langle e_1-e_2,\,e_2-e_3\rangle=A_2$.  For the node,
$\Mil\simeq T^*S^3$ has $H_2=0$.
\end{proof}

\begin{remark}\label{rem:xi}
The three kernels are the lattices $\Xi(V)$ of
Kuznetsov--Prokhorov: \cite[Theorem~1.9(i)]{KP23} states
$\Cl(V)\cong\Xi(V)^{\perp}\subset\Cl(\Sb)$ with $\Xi(V)$ negative
definite of rank $10-d-\rho(V)$, giving $A_1$, $A_1$, $A_2$ for
$V_7$, $(\PP^1)^3$, $W_6$, respectively.
\end{remark}

\subsection{Which sweep for which component}\label{sec:match}
The one-dimensional (line) component of the hexagon base sweeps to
$(\PP^1)^3$ and the two-dimensional (plane) component to $W_6$: this is
stated in \cite[Proposition~2.2]{Petracci20}, follows from
\cite[(7.1.5)]{Altmann95} (the total spaces over the two components are
the cones over $(\PP^1)^3$ and over $\PP^2\times\PP^2$, and the
hyperplane section of $\PP^2\times\PP^2$ in the Segre embedding is $W_6$).
The corresponding Betti numbers are given in
\cite[Proposition~3.5]{KP21}.  Theorem~\ref{thm:main} uses only the
nodes and the degree-$7$ branch.  The matching is used in
Lemma~\ref{lem:D} for the degree-$6$ line branch, and through it in
Theorem~\ref{thm:necessity}(3) and in the degree-$6$ assertion of
Theorem~\ref{thm:partner}.

\section{Vanishing periods}\label{sec:periods}

This section contains the analytic mechanism.  Throughout,
$\cX\to\Disk$ is a one-parameter smoothing of a connected compact complex
threefold $X$ with $\omega_X\cong\cO_X$, $h^1(\cO_X)=0$ and isolated
singularities of the types in Theorem~\ref{thm:necessity-intro}.

\subsection{A relative holomorphic \texorpdfstring{$3$}{3}-form}\label{sec:relform}
Flatness and the Gorenstein property of all fibers make
$\cX\to\Disk$ a Gorenstein morphism after shrinking the disk, so the
relative dualizing sheaf $\omega_{\cX/\Disk}$ is invertible and
compatible with base change \cite{Conrad00} (in the analytic category this
also follows from adjunction, as in Remark~\ref{rem:partialform}).  We first prove
$c_1(\omega_{X_s})=0$ integrally.  Choose the disjoint good
representatives and exterior $W$ of Section~\ref{sec:decomp}.  Over a
neighborhood of $W$ the family is a smooth proper submersion with
boundary and hence is $C^\infty$-trivial.  Since
$\omega_X|_W\cong\cO_W$, the restriction of
$c_1(\omega_{\cX/\Disk})$ to the corresponding exterior in $X_s$
vanishes.  On each local piece it also vanishes: choose first a Stein
contractible neighborhood $N_p$ of $(p,0)$ in $\cX$, on which the
invertible sheaf $\omega_{\cX/\Disk}$ is trivial, and then the good
representative $U_p$ inside $N_p$.

For every boundary link in this paper, $H^1(L_p;\ZZ)=0$.  This is
immediate for the nodal link $S^2\times S^3$.  For a cone point, $L_p$
is the circle bundle of $\cO_S(K_S)$ over the simply connected del
Pezzo surface $S$; the homotopy sequence shows that $\pi_1(L_p)$ is
the cyclic group $\ZZ/d$, where $d$ is the divisibility of the Euler
class $K_S\neq0$, because the connecting map
$\pi_2(S)\to\pi_1(S^1)=\ZZ$ is evaluation against $K_S$.  Hence
$H^1(L_p;\ZZ)=\operatorname{Hom}(\ZZ/d,\ZZ)=0$ for every del Pezzo
cone point, whether or not $K_S$ is primitive.  (It is primitive for
$\dP_6$, $\dP_7$ and $\mathbb F_1$, where $L_p$ is simply connected,
but not for $\PP^1\times\PP^1$.)  The Mayer--Vietoris
sequence for $X_s=W\cup\bigsqcup_p\Mil_p$ therefore injects
\[
 H^2(X_s;\ZZ)\lhook\joinrel\longrightarrow
 H^2(W;\ZZ)\oplus\bigoplus_pH^2(\Mil_p;\ZZ).
\]
The preceding local vanishings give $c_1(\omega_{X_s})=0$.
Semicontinuity gives $h^1(\cO_{X_s})=0$; the exponential sequence then
embeds $\Pic(X_s)$ into $H^2(X_s;\ZZ)$ and forces
$\omega_{X_s}\cong\cO_{X_s}$.  (The hypothesis
$\omega_X\cong\cO_X$ cannot be dropped: the cone over a smooth quadric
surface in $\PP^4$ is a one-node threefold smoothed by the family of
quadric threefolds, and the vanishing cycle of that smoothing
\emph{is} null-homologous; that example satisfies $h^1(\cO_X)=0$ and
is excluded here exactly because $\omega_X\cong\cO_X(-3)$ is
nontrivial.  The hypothesis $h^1(\cO_X)=0$ enters only through the
exponential sequence in the previous sentence.)  Constants and upper
semicontinuity give $h^0(\cO_{X_s})=1$ after shrinking, hence
$h^0(\omega_{X_s})=1$.  By Grauert's base-change theorem
\cite{Grauert60}, $\pi_*\omega_{\cX/\Disk}$ is then a line bundle on the shrunk
disk; we fix a trivializing section $\Omega$, whose restriction
$\Omega_s$ generates the invertible sheaf $\omega_{X_s}$ and is
therefore a nowhere-vanishing holomorphic $3$-form on the smooth locus
of $X_s$ for $s\neq0$ and a generator of $\omega_X$ near each singular point at
$s=0$.

\subsection{Nodes}\label{sec:nodeperiod}
\begin{lemma}\label{lem:P}
Let $p$ be an ordinary double point of $X$.  Then the vanishing-cycle
class $\delta_p(s)\in H_3(X_s;\QQ)$ is nonzero for all small
$s\neq0$.
\end{lemma}

\begin{proof}
The germ $(X,p)$ is $\{xy-zw=0\}\subset(\CC^4,0)$, with miniversal
deformation $\{xy-zw=u\}$ over the $u$-line
\cite{KasSchlessinger72}.  The family induces a classifying map
$s\mapsto\tilde u(s)$, holomorphic with $\tilde u(0)=0$ (only its
nonvanishing at each fixed $s\neq0$ is used, which is independent of
choices), and $X_s$
is smooth near $p$ exactly when $\tilde u(s)\neq0$; the hypothesis
gives $\tilde u(s)\neq0$ for all small $s\neq0$.  The pullback total
space is the hypersurface $\{xy-zw=\tilde u(s)\}$, Gorenstein even at
its singular points, and adjunction exhibits $\omega_{\cX/\Disk}$
near $p$ as generated by the relative Gelfand--Leray form
$\GL_u=\operatorname{Res}\bigl[dx\,dy\,dz\,dw/(xy-zw-u)\bigr]$.
Hence $\Omega=g\cdot\GL$ with $g$ holomorphic on the total-space germ
and $g(p,0)\neq0$; extend $g$ to an ambient holomorphic function so
that $g(p,s)$ is defined.

After the linear change putting $xy-zw$ into the form
$\sum_{i=1}^4v_i^2$, the vanishing cycle over $u$ is the real slice
$\delta_u=\sqrt u\cdot S^3$, whose class is single-valued: in fiber
dimension three the Picard--Lefschetz transvection fixes $\delta$
(as $\langle\delta,\delta\rangle=0$), and the antipodal ambiguity of
$\sqrt u$ acts on $S^3$ with degree $+1$.  Under
$v\mapsto\lambda v$ one has $u\mapsto\lambda^2u$ and
$\GL\mapsto\lambda^2\GL$, so the local period is exactly linear:
\[
\int_{\delta_u}\GL_u\;=\;c_0\,u,
\qquad
c_0=4\int_{S^3}\frac{dv_1\,dv_2\,dv_3}{2v_4}=4\pi^2\neq0
\]
up to orientation conventions, the factor $4$ being the Jacobian
determinant of the linear change of coordinates.  For the twisted form, expanding the
ambient extension $g=g(p,s)+O(|v|)$ and using $|v|=|u|^{1/2}$ on
$\delta_u$ gives
\[
F(u,s):=\int_{\delta_u}g\cdot\GL_u
= u\,\bigl(c_0\,g(p,s)+O(|u|^{1/2})\bigr),
\]
so $F(u,s)=u\cdot(\text{unit})$ near $(0,0)$; here $u,s$ are
independent until the substitution $u=\tilde u(s)$.

Suppose $\delta_p(s_0)=0$ in $H_3(X_{s_0};\QQ)$ for some small
$s_0\neq0$.  A nonzero integer multiple of the cycle then bounds, and
$\Omega_{s_0}$ is closed, so
$\int_{\delta_p(s_0)}\Omega_{s_0}=0$.  But this period equals
$F(\tilde u(s_0),s_0)=\tilde u(s_0)\cdot(\text{unit})\neq0$, a
contradiction.
\end{proof}

\subsection{Del Pezzo cone points}\label{sec:dpperiod}
For the cone germs the same strategy applies, with the local model
family replaced by the sweep and the explicit period computed on the
sweep complement.  Both relevant complements admit a uniform toric
description.  Let $V$ be $(\PP^1)^3$ or $\Bl_{pt}\PP^3$ with its fan
in $N\cong\ZZ^3$, and let $m\in N^{\vee}$ be the character $(1,1,1)$,
respectively $(1,1,-1)$ in the basis where the fan of
$\Bl_{pt}\PP^3$ has rays $e_1,e_2,e_3$, $r=-e_1-e_2-e_3$ and
$e_1+e_2+e_3$.  In both cases $\langle m,v\rangle=\pm1$ on every ray,
the closure $\Sb$ of $\{\chi^m=1\}$ is a copy of the appropriate del
Pezzo surface with $[\Sb]=-\tfrac12K_V$, and
\begin{equation}\label{eq:omegaM}
\Omega_\Mil=\frac{w}{(w-1)^2}\,
\frac{dx}{x}\wedge\frac{dy}{y}\wedge\frac{dz}{z},
\qquad w=\chi^m,
\end{equation}
is the generator of $\omega(V\smallsetminus\Sb)\cong\cO$, unique up to
scale: the order of \eqref{eq:omegaM} along a boundary divisor $D_v$
is $\langle m,v\rangle-2\min(\langle m,v\rangle,0)-1=0$, and along
$\Sb$ it is $-2$.  Here uniqueness is algebraic.  If
$f\in\cO(V\smallsetminus\Sb)^\times$, then
$\operatorname{div}_V(f)=n\Sb$ for some $n\in\ZZ$; since
$\Pic(V)$ is torsion-free and $[\Sb]\neq0$, one has $n=0$, and a
rational function without zeros or poles on the projective variety
$V$ is constant.

Every sweep pair $(V,\Sb)$ of Section~\ref{sec:sweeps} with $V=\Bl_{pt}\PP^3$
or $V=(\PP^1)^3$ is isomorphic to the toric pair just described.  On
$\Bl_{pt}\PP^3$ the smooth members of $|H|=|2H_{\PP^3}-E|$ are the strict
transforms of the smooth quadrics through the blown-up point.  The group
$\mathrm{PGL}_4$ acts transitively on smooth quadrics, and the automorphism
group of a smooth quadric acts transitively on its points, so the
automorphisms of $\PP^3$ fixing the point act transitively on these members.
On $(\PP^1)^3$, with homogeneous coordinates $[x_0:x_1]$, $[y_0:y_1]$,
$[z_0:z_1]$ on the three factors (so that $x=x_1/x_0$, $y=y_1/y_0$,
$z=z_1/z_0$ above), a member of $|H|$ is the zero locus of a trilinear form
$z_0A_0(x,y)+z_1A_1(x,y)$ with $A_0,A_1$ bilinear.  It is smooth exactly when
the binary quadratic form $\det(\lambda A_0+\mu A_1)$ has nonzero
discriminant, Cayley's hyperdeterminant \cite[Chapter~14, \S\S1 and~4]{GKZ94}.  A linear
change of $(z_0,z_1)$ then moves the two roots to $\lambda=0$ and $\mu=0$, so
that $A_0$ and $A_1$ have rank one with invertible sum, and a linear change
of $x$ and of $y$ brings the form to $x_0y_0z_0-x_1y_1z_1$, the equation of
the closure of $\{\chi^m=1\}$.

\begin{lemma}\label{lem:D1}
In each of the two complements $\Mil=V\smallsetminus\Sb$ above there is a
closed $3$-cycle $T\subset \Mil$ with $\int_T\Omega_\Mil=\pm4\pi^2$.
Consequently $[\Omega_\Mil]\neq0$ in $H^3(\Mil;\CC)$, and since
$b_3(\Mil)=1$, the period pairing
$H_3(\Mil;\QQ)\to\CC$, $\nu\mapsto\int_\nu\Omega_\Mil$, is injective.
\end{lemma}

\begin{proof}
Choose an edge $\langle v_+,v_-\rangle$ of the fan with
$\langle m,v_\pm\rangle=\pm1$ and an adjacent maximal cone
$\langle v_+,v_-,n_0\rangle$ with $\langle m,n_0\rangle=+1$, such that the
maximal cone $\langle v_+,v_-,n_1\rangle$ on the other side of the cone
$\langle v_+,v_-\rangle$ has $\langle m,n_1\rangle=-1$.  This last condition
keeps the cap $T_3$ below off $\Sb$; without it the cap can meet $\Sb$.  The
two cases are
\begin{center}
\begin{tabular}{lccc}
\toprule
$V$ & $(v_+,v_-,n_0;n_1)$ & transition to the $n_1$ chart & $c$\\
\midrule
$(\PP^1)^3$ & $(e_1,-e_2,e_3;-e_3)$ & $(p,\,q,\,1/z)$ & $0$\\
$\Bl_{pt}\PP^3$ & $(e_1,r,e_2;e_3)$ & $(p/z,\,q/z,\,1/z)$ & $-1$\\
\bottomrule
\end{tabular}
\end{center}
where $c$ is the $z$-exponent of the transition monomial for $p'$.  In the
unimodular chart with coordinates $(p,q,z)$ dual to $(v_+,v_-,n_0)$ one has
$w=pq^{-1}z$, $\Sb=\{t=0\}$ with $t=q-pz$, and
$\Omega_\Mil=\pm\,dp\wedge dq\wedge dz/t^{2}$.
Fix $\varepsilon>0$, $R>1$ and set
\begin{align*}
T_1&=\{p=0,\ q=\varepsilon Re^{i\psi},\ |z|\le R\},\\
T_2&=\{p=r\varepsilon e^{i\theta},\
q=-(1-r)\varepsilon Re^{i(\theta+\varphi)},\
z=Re^{i\varphi}\},\qquad r\in[0,1],\\
T_3&=\{p'=\varepsilon R^{c}e^{i\psi},\ q'=0,\ |z'|\le 1/R\}
\quad\text{in the chart at the opposite corner }n_1,
\end{align*}
where $(p',q',z')$ are the dual coordinates of $(v_+,v_-,n_1)$ and $c$ is as in
the table.  The corner-chart cap is necessary: in the
$\Bl_{pt}\PP^3$ case the naive closure of
$\{|p|=\varepsilon,q=0,|z|\ge R\}$ collapses onto the torus-fixed
point of the corner chart, which \emph{lies on} $\Sb=\{p'=q'z'\}$;
the cap $T_3$ instead keeps $|p'-q'z'|=\varepsilon R^{c}>0$.  On
$T_2$, $t=-\varepsilon Re^{i(\theta+\varphi)}$ has
constant modulus; the $r=0$ boundary is the boundary torus of $T_1$
and the $r=1$ boundary maps under the transition onto
$\partial T_3$, so with orientations of the caps chosen to cancel
$\partial T_2$, $T=T_1+T_2+T_3$ is a closed $3$-cycle in $\Mil$.

The pullback of $\Omega_\Mil$ to $T_1$ and $T_3$ vanishes identically
($dp\mapsto0$, respectively $dq'\mapsto0$, while $\Omega_\Mil$ is
regular there), and on $T_2$ the pullback of
$dp\wedge dq\wedge dz/t^2$ is the \emph{constant} form
$dr\wedge d\theta\wedge d\varphi$: writing it out, the complex
Jacobian of the parametrization equals $t^2$ exactly.
Hence $\int_T\Omega_\Mil=\int_0^1\!\int_0^{2\pi}\!\int_0^{2\pi}
dr\,d\theta\,d\varphi=4\pi^2$ up to the orientation sign.

Finally $b_3(\Mil)=1$: from the pair sequence of $(V,\Mil)$ and
$H^3(V)=0$, $b_3(\Mil)$ is the corank of the restriction
$H^2(V;\QQ)\to H^2(\Sb;\QQ)$, which is one in both cases.
\end{proof}

\begin{lemma}\label{lem:D}
Let $p$ be a singular point of $X$ whose germ is the anticanonical
cone over $\dP_7$, or over $\dP_6$ with the induced local smoothing
through the one-dimensional component.  Then
$H_3(\Mil_p;\QQ)\;(\cong\QQ)\to H_3(X_s;\QQ)$ is injective for all small
$s\neq0$.
\end{lemma}

\begin{proof}
By Section~\ref{sec:miniversal} the induced arc lies in a single
reduced component $c$ of the miniversal base of the germ, off the
local discriminant for $s\neq0$, and by Section~\ref{sec:sweeps} the
Milnor fiber is the corresponding sweep complement, which for the two
branches in the statement is isomorphic to one of the complements of
Lemma~\ref{lem:D1} by the paragraph preceding that lemma.  The component carries Pinkham's
$\CC^*$-equivariant swept family with parameters of weight one
(Section~\ref{sec:miniversal}: $T^1$ sits in the Gorenstein degree);
let $\Omega^{\mathrm{loc}}$ be an equivariant algebraic generator of
its relative dualizing sheaf, homogeneous of some integral weight
$a$.  To construct it, work in the graded local ring at the vertex.
The relative dualizing module is invertible, and graded Nakayama
selects from any algebraic local generator a homogeneous component
that still generates.  Call $U$ the open locus on which that homogeneous
generator generates; $U$ is $\CC^*$-stable, being the nonvanishing locus of a
homogeneous section, and it contains the vertex.  Then $U$ contains the whole
fiber $b=1$.  Let $x$ be a point of that fiber.  Scaling the positive-weight
cone coordinates and the weight-one parameter to zero carries $x$ to the vertex,
so the vertex lies in the closure of $\CC^*\cdot x$; since $U$ is open and
contains the vertex, some $\lambda\cdot x$ lies in $U$, and since $U$ is
$\CC^*$-stable, $x$ does.  Thus the restriction to $b=1$ is an
algebraic
generator of $\omega_\Mil$, hence a constant multiple of $\Omega_\Mil$ by
the unit calculation following \eqref{eq:omegaM}; rescale it so that
$\Omega^{\mathrm{loc}}|_{b=1}=\Omega_\Mil$.
For a horizontal family $\nu$ of classes in $H_3(\Mil;\QQ)$,
\[
\int_{\nu(b)}\Omega^{\mathrm{loc}}_b\;=\;c(\nu)\cdot\chi(b),
\]
with $\chi$ the weight-$a$ monomial in the component parameter and
$c(\nu)=\int_\nu\Omega_\Mil$: single-valuedness holds because the
$\CC^*$-action is a group action on the total space with parameter
weight one, so transport once around the base circle is the identity;
and $c(\cdot)$ is linear.  Writing
$\Omega=g\cdot\Omega^{\mathrm{loc}}$ near $p$ with $g$ a unit as in
Section~\ref{sec:relform}, let $b=b(s)$ be the actual parameter of the
component.  The cycle obtained from the fixed compact cycle of
Lemma~\ref{lem:D1} by the $\CC^*$-action lies in the local Milnor
piece and converges uniformly to $p$, because all cone coordinates
have positive weights.  Hence $g|_{\nu(s)}=g(p,0)+o(1)$ uniformly,
while homogeneity factors the same monomial $\chi(b(s))$ from the main
term and the error.  The global period is therefore
\[
\int_{\nu(s)}\Omega_s
=\chi(b(s))\bigl(c(\nu)\,g(p,0)+o(1)\bigr).
\]

If $\nu\neq0$ in $H_3(\Mil;\QQ)$ had $\nu_*=0$ in $H_3(X_s;\QQ)$ for
arbitrarily small $s\neq0$, then a rational multiple of $\nu(s)$
would bound and the closed form $\Omega_s$ would have zero period,
while $\chi(b(s))\neq0$ off the discriminant; letting $s\to0$ along such
values gives $c(\nu)\,g(p,0)=0$ and, since $g(p,0)\neq0$,
$c(\nu)=0$, contradicting the injectivity of the period pairing in
Lemma~\ref{lem:D1}.  (That $H_3(\Mil;\QQ)$ has rank one is what converts
this class-by-class exclusion into injectivity for all small $s$; a
higher-rank branch would require a uniform argument.)
\end{proof}

\section{Surgery membership}\label{sec:surgery}

\subsection{The decomposition}\label{sec:decomp}
Let $p_1,\dots,p_r$ be the singular points of $X$, with disjoint
closed cone neighborhoods $U_p$ chosen inside Milnor balls, and
$W=X\smallsetminus\bigcup_p\operatorname{int}U_p$, so
$\partial W=\bigsqcup_pL_p$.  Good representatives exist for
arbitrary isolated singularities
(Greuel~\cite[Definition~2.12, Lemma~2.13, Theorem~2.15]{Greuel19}),
and for small $s$ the family is smooth, hence $C^\infty$-trivial, over
a neighborhood of $W$; therefore
\[
X_s\ \cong\ W\ \cup_{\bigsqcup L_p}\ \bigsqcup_p\Mil_p,
\qquad
\Xh\ \cong\ W\ \cup_{\bigsqcup L_p}\ \bigsqcup_pN_p,
\]
where $\Mil_p$ is the Milnor fiber of the induced local smoothing and
$N_p$ a neighborhood of the exceptional set of the fixed crepant
resolution.  Both gluings are along the same links.  Only the
$C^\infty$-triviality over a neighborhood of $W$ is used here, so the
same decomposition of $X_s$ along $\bigsqcup_pL_p$ holds for an
arbitrary one-parameter deformation, with $X_s\cap U_p$ in place of
$\Mil_p$ at any point whose germ is not smoothed; this is the form
Section~\ref{sec:partialmem} uses.

\begin{lemma}[membership]\label{lem:M}
For every $\gamma\in H_3(X_s;\QQ)$, the Mayer--Vietoris boundary
$\partial\gamma\in\bigoplus_pH_2(L_p;\QQ)$ lies in the space $K(X)$
of Theorem~\ref{thm:necessity-intro}: its components die in every
$H_2(\Mil_p;\QQ)$, and their sum dies in $H_2(W;\QQ)$, hence in
$H_2(\Xh;\QQ)$.  Moreover, at a node $q$ the $L_q$-component of
$\partial\gamma$ is a fixed nonzero multiple of
$(\gamma\cdot\delta_q)$ times the generator, and at a cone point $p$
the boundary map identifies
$H_3(\Mil_p,L_p;\QQ)\xrightarrow{\ \sim\ }\Rsm_p$.
\end{lemma}

\begin{proof}
The first statement is Mayer--Vietoris exactness for
$X_s=W\cup\bigsqcup \Mil_p$, together with $W\subset\Xh$.  At a node,
$\Mil\simeq T^*S^3$: $H_3(\Mil,L;\QQ)\cong H^3(\Mil;\QQ)=\QQ$ is generated by
a cotangent fiber $F$ with $\partial F$ the $S^2$-generator of
$H_2(S^2\times S^3)$ and $F\cdot\delta=1$; excision identifies the
$L_q$-component of $\partial\gamma$ with
$\partial$ of the image of $\gamma$ in $H_3(\Mil_q,L_q)$, which equals
$(\gamma\cdot\delta_q)F$ by the perfect pairing.  At a cone point, exactness of the pair sequence makes the image of
$\partial\colon H_3(\Mil,L;\QQ)\to H_2(L;\QQ)$ equal to
$\ker(H_2(L;\QQ)\to H_2(\Mil;\QQ))=\Rsm_p$, while Lefschetz duality gives
$\dim H_3(\Mil,L;\QQ)=b_3(\Mil)$, which by Section~\ref{sec:sweeps} equals
$1$, $1$, $2$ in the three branches, the rank of $\Rsm_p$ computed in
Proposition~\ref{prop:kernels}.  (The comparison of the actual Milnor fiber
with the sweep complement as a manifold with boundary, which this uses, is
Lemma~\ref{lem:boundary} below.  When $b_3(\Mil)=1$, injectivity of
$\partial$ needs no comparison: the map $H_3(\Mil)\to H_3(\Mil,L)$ is
Lefschetz dual to a skew form on a line, hence zero.)  A surjection between
spaces of equal finite dimension is an isomorphism, so $\partial$ is
injective with image $\Rsm_p$; equivalently $H_3(\Mil)\to H_3(\Mil,L)$ is
zero, $H_3(L)\to H_3(\Mil)$ is surjective, and the skew intersection form of
$H_3(\Mil;\QQ)$ vanishes.
\end{proof}

\subsection{Boundary compatibility}\label{sec:boundary}
Lemma~\ref{lem:M} uses the subspaces
$\Rsm_p\subset H_2(L_p)$ through the \emph{sweep-side} description of
Proposition~\ref{prop:kernels} (the unit normal circle bundle of
$\Sb\subset V$), while the map to $H_2(\Xh)$ and all coordinates in
Section~\ref{sec:example} use the \emph{resolution-side} description
(the projection $L_p\to E_p$ of the cone link onto the exceptional del Pezzo,
the circle-bundle projection recalled after Lemma~\ref{lem:boundary}).  These are a
priori different identifications of $H_2(L_p;\QQ)$ with
$K^{\perp}$, and an arbitrary boundary diffeomorphism need not preserve
either marking.  Their compatibility is therefore a statement, not a
formality.  It appears not to be available in the literature:
the identification of the germ Milnor fiber with the sweep complement
is standard as \emph{open} manifolds
(\cite[Proposition~4.2]{PP15}), but the manifold-with-boundary form,
compatible with the canonical link identification, is what the
following lemma supplies.

\begin{lemma}[boundary compatibility]\label{lem:boundary}
Let $(X,p)$ be analytically isomorphic to the anticanonical cone over
the del Pezzo surface $S$, smoothed by a one-parameter family through
the component $c$ with sweep $V\supset\Sb$.  Then the boundary
identification $\partial \Mil_p\cong L_p$ arising from the surgery
decomposition is such that the sweep-side projection
$H_2(L_p;\QQ)\to H_2(\Sb;\QQ)$ and the circle-bundle projection
$H_2(L_p;\QQ)\to H_2(E_p;\QQ)$ agree after composing the former with the
isometry $H_2(\Sb;\ZZ)\to H_2(E_p;\ZZ)$ induced by an isomorphism
$\Sb\cong E_p$, which carries $K_{\Sb}$ to $K_{E_p}$.  Consequently $\Rsm_p$,
read in resolution-side coordinates, is the root subspace of
Proposition~\ref{prop:kernels}.
\end{lemma}

Here the germ-local crepant resolution of $C(S)$ is the blowup of the
vertex, its exceptional divisor is the zero section $E_p\cong S$ of the
total space of $\cO_S(K_S)$, the link $L_p$ is the unit circle bundle of
that line bundle, and the \emph{circle-bundle projection} is the bundle
projection $L_p\to E_p$; this is the identification used in
\eqref{eq:linkH2}.  This local form is part of the standing hypotheses and
is not automatic for a crepant resolution: a triangulation of a pentagon or
hexagon that cuts off a unimodular corner gives a toric crepant resolution
with a smaller toric surface over $p$ together with an exceptional curve.

\begin{proof}
Let $\mathcal Y_c\to B_c$ be the $\CC^*$-equivariant
Altmann--Pinkham family on the reduced component $c$.  All component
parameters have weight one and all affine cone coordinates have positive
weights, so an equivariant closed embedding
$\mathcal Y_c\hookrightarrow\AAf^N\times B_c$ with fiber coordinate weights
$w_1,\dots,w_k>0$ exists, and we let $\overline{\mathcal Y}_c\to B_c$ be the
fiberwise closure of $\mathcal Y_c$ in
$\PP(1,w_1,\dots,w_k)\times B_c$.  (For the del Pezzo surfaces of degrees~$6$
and~$7$ the anticanonical ring is generated in degree one, so all $w_i$ may be
taken equal to~$1$ and the ambient space is $\PP^k\times B_c$; general weights
are kept only for uniformity of notation.)  This is Pinkham's sweeping out of
the cone \cite[(4.2),\,(5.1)]{Pinkham74} carried out in families over $B_c$:
the divisor at infinity is cut by the homogenizing coordinate, on which
$\CC^*$ acts trivially, and is
$\Proj$ of the graded coordinate ring of the cone, hence the fixed family
$\Sb\times B_c$, independent of $b$ because the parameters have positive
weight.  For $b$ off the discriminant the pair
$(\overline{\mathcal Y}_b,\Sb)$ is the polarized sweep pair $(V,\Sb)$ of
Proposition~\ref{prop:kernels}, which is where its smoothness comes from.  Choose an equivariant affine embedding,
write $w_1,\dots,w_k>0$ for the coordinate weights and, for a common
multiple $d$ of them, use the weighted radius
\[
 \rho(z)=\left(\sum_{i=1}^k|z_i|^{2d/w_i}\right)^{1/(2d)},
 \qquad \rho(r\cdot z)=r\rho(z)\quad(r>0).
\]
Fix a sufficiently small weighted Milnor radius $\varepsilon$ and a
small sphere of radius $\eta$ in $B_c$.  Every sufficiently small
parameter $\beta\in c$ off the discriminant is uniquely
$\beta=t b$ with $t>0$ and $\lVert b\rVert=\eta$.  The discriminant is
homogeneous, so $b$ is also off it.  The action by $t^{-1}$ gives a
diffeomorphism of pairs
\begin{equation}\label{eq:compactcore}
 \left(\mathcal Y_\beta\cap\{\rho\leq\varepsilon\},
       \mathcal Y_\beta\cap\{\rho=\varepsilon\}\right)
 \ \cong\
 \left(K_{t,b},\Sigma_{t,b}\right),
 \quad
 K_{t,b}:=\mathcal Y_b\cap\{\rho\leq\varepsilon/t\}.
\end{equation}
The pair on the left is the actual compact Milnor fiber and its
boundary, not merely its interior.

We next identify the pair on the right with the sweep compact core.
Near the fixed divisor $S\subset\overline{\mathcal Y}_b$, choose a
normal coordinate $q$, and pass to the real oriented blowup along
$S$.  An affine coordinate of weight $w_i$ has the form
\[
 z_i=q^{-w_i}\bigl(a_i(y,b)+O(q)\bigr),
\]
where $y\in S$ and the leading coefficients do not vanish
simultaneously.  Consequently
\begin{equation}\label{eq:radiusatinfinity}
 \rho=|q|^{-1}A(y,\arg q,|q|,b)
\end{equation}
with $A$ smooth and strictly positive on the oriented boundary.  For
$t$ small, the equation $\rho=\varepsilon/t$ therefore has a unique
solution for $|q|$ and is a smooth radial graph over the unit normal
circle bundle of $S$.  Normal radial flow identifies this graph with
the boundary of a fixed tube of $S$, preserving the projection to
$S$, and extends across the intervening collar.  Thus
\eqref{eq:compactcore} continues to a diffeomorphism of pairs
\[
 (\Mil_p,\partial \Mil_p)\ \cong\
 \bigl(\overline{\mathcal Y}_b\smallsetminus
       \operatorname{int}\mathrm{Tube}(S),
       \partial\mathrm{Tube}(S)\bigr)
\]
that carries the actual boundary inclusion to the tube-boundary
inclusion.

By the sweep description of Section~\ref{sec:sweeps}, the compactified
smooth pair $(\overline{\mathcal Y}_b,S)$ is the corresponding
polarized sweep pair $(V,\Sb)$; along a connected smooth branch the
pair identifications may be transported by Ehresmann, using a
trivializing vector field that is tangent to the fixed family
$S\times B_c$ and horizontal along it, so that $S$ is fixed pointwise.

It remains to identify the marking of that boundary.  Apply the same
construction to the radial family $\mathcal Y_{\tau b}$ for
$0<\tau\leq t$.  On the real oriented blowup write $r=|q|$ and
$\theta=\arg q$.  By~\eqref{eq:radiusatinfinity} the boundary of
$K_{\tau,b}$ is $\{r=(\tau/\varepsilon)A(y,\theta,r,b)\}$, and the implicit
function theorem gives a unique solution
$r^*(\tau,y,\theta)=(\tau/\varepsilon)A(y,\theta,0,b)+O(\tau^2)$, smooth up to
$\tau=0$.  Let $\Psi_\tau(y,\theta)\in\mathcal Y_{\tau b}\cap\{\rho=\varepsilon\}$
be the image under the action of $\tau$ of the point with coordinates
$(y,\theta,r^*)$.  Its affine coordinates are
$(\tau/r^*)^{w_i}e^{-i\theta w_i}\bigl(a_i(y,b)+O(r^*)\bigr)$, which converge
smoothly as $\tau\to0$ to a point of the link of the central cone.  Thus
$\Psi$ trivializes the boundary bundle over the segment $[0,t]\cdot b$, and
for every $\tau$ the tube projection composed with $\Psi_\tau$ is
$(y,\theta)\mapsto y$.  At $\tau=0$ the projection sends a nonzero point
of the central cone to its projective class in $S$, while the normal
circle is its $S^1$-orbit.  It is therefore exactly the circle-bundle
projection $L_p\to S$ of the cone link.

It remains to compare this radial transport with the identification
$\partial \Mil_p\cong L_p$ that Section~\ref{sec:decomp} fixes, which is not a
free choice.  The two live over different parameter spaces, so name the map
between them: let $\beta\colon(\Disk,0)\to(B_c,0)$ classify the induced
local deformation, so that after shrinking the local family is
$\beta^{*}\mathcal Y_c$, and take the good representative $U_p$ of
Section~\ref{sec:decomp} with $\partial U_p$ corresponding to
$\{\rho=\varepsilon\}$.  The comparison is in three steps: the boundary is a
locally trivial bundle over the whole base, discriminant included; two
trivializations of it over the disc induce the same map on $H_2$; and the
radial transport is one such trivialization.

First, $\{\rho=\varepsilon\}\cap\mathcal Y_b$, for $b\in B_c$, is a locally
trivial $C^\infty$ bundle over all of $B_c$, discriminant included.  The
relative singular locus is proper over the compact ball
$\{\rho\le\varepsilon\}$ and meets the central fiber only at the vertex, so
for small $b$ it stays inside $\{\rho<\varepsilon\}$; the total space is
smooth near $\{\rho=\varepsilon\}$; and $\mathcal Y_b$ meets
$\{\rho=\varepsilon\}$ transversally, which holds at $b=0$ by
$\rho(r\cdot z)=r\rho(z)$ and is open in $b$ by compactness.  Ehresmann
applies.

Second, fix a trivialization $T$ of that bundle.  Then $\beta^{*}T$ and the
trivialization of Section~\ref{sec:decomp} are two trivializations of the
same bundle over $\Disk$.  Two such identify a fiber with the fiber over
$0$ by maps differing by $g(0)^{-1}g(s)$ for some $g$ into the
diffeomorphism group of $L_p$; a path from $0$ to $s$ makes
$t\mapsto g(\gamma(t))$ an isotopy from $g(0)$ to $g(s)$, so the
discrepancy is isotopic to the identity and acts trivially on $H_2$.  Only
path-connectedness is used.

Third, the radial transport trivializes the same bundle over the ray
through $\beta(s)$, so on that ray it agrees with $T$ on $H_2$ for the same
reason.  The sweep-side tube projection and the resolution-side
circle-bundle projection therefore induce the same map on $H_2$, up to the
isometry of
the statement.

Finally the $A_1$ and $A_2$ summands for $\dP_6$ and the unique
$(-2)$-line for $\dP_7$ are intrinsic to the polarized lattice, so
they correspond.  Since the displayed diffeomorphism is one of pairs,
it computes the kernel of the actual inclusion
$H_2(\partial \Mil_p)\to H_2(\Mil_p)$, as required.
\end{proof}

\subsection{Deformations that leave germs singular}\label{sec:partialmem}
Theorem~\ref{thm:necessity-intro} concerns smoothings.  A threefold carrying a
germ with no smoothing at all admits none, but it may still admit
deformations that smooth its other germs, and the surgery argument
constrains those as well.  We record the form of Lemma~\ref{lem:M}
that Section~\ref{sec:partner} uses.  Call an isolated singularity
\emph{reduced-rigid} if the reduction of its miniversal base is a
single point; this does not assert $T^1=0$.  The anticanonical cone
over $\mathbb F_1=\Bl_1\PP^2$, one of the two del Pezzo surfaces of
degree~$8$ (not $\PP^1\times\PP^1$, whose cone is smoothable), is
reduced-rigid: its quadrilateral, with edge vectors
$(1,0),(1,2),(-2,-1),(0,-1)$, admits no nontrivial Minkowski
decomposition \emph{into lattice polygons}, although its cone of
rational Minkowski summands is two-dimensional, so that $\dim T^1=1$
and the miniversal base is
$\operatorname{Spec}\CC[\varepsilon]/(\varepsilon^2)$
\cite[(9.1)(iv),\,(9.2)]{Altmann97}; see also
Gross~\cite[\S5.5]{Gross97} and the example of \cite[\S2]{Gross97} (Example~2.8
of the 1995 preprint, Example~2.19 of the 2025 revision
arXiv:alg-geom/9506022v2), and \cite{PaperI}.

\begin{lemma}\label{lem:rigid}
Let $(Y,0)$ be reduced-rigid and let $\cX\to\Disk$ be a
one-parameter deformation of a complex space $X$ with
$(X,p)\cong(Y,0)$.  Then the induced deformation of the germ at $p$ is
trivial.  Consequently there are a neighborhood $U$ of $p$ in $\cX$
and an $\varepsilon>0$ with $U\cong U_0\times\Disk_\varepsilon$ over
$\Disk_\varepsilon$, where $U_0$ is a neighborhood of $p$ in $X$, so that for $0<|s|<\varepsilon$ the fiber
$X_s\cap U$ has a singular point $q_s$, with $q_s\to p$ and
$(X_s,q_s)\cong(Y,0)$.
\end{lemma}

\begin{proof}
Near $p$ the family is induced from the miniversal deformation by a
holomorphic map $c\colon(\Disk,0)\to(B,0)$.  The analytic local ring
of a disk is a domain, hence reduced, so $c$ factors through
$B_{\mathrm{red}}$, a reduced point; thus $c$ is constant, and the
pullback of the miniversal family along a constant map is the product
$(Y,0)\times(\Disk,0)$.  The final statement is the passage from the
germ at $(p,0)$ to fibers, by choosing representatives of that
product decomposition.
\end{proof}

\begin{lemma}[membership, with germs left singular]\label{lem:Mpartial}
Let $X$ and $\Xh$ be as in the setting fixed in
Section~\ref{sec:intro} before Theorem~\ref{thm:necessity-intro}, except that
$X$ may carry in addition finitely many anticanonical cones over
$\mathbb F_1$, with $\Xh\to X$ the blowup of the vertex, with exceptional
surface $\mathbb F_1$, at each of them; no smoothing of $X$ is assumed to exist.  Let $\cX\to\Disk$ be
any one-parameter deformation of $X$ and let
\[
\Sigma\ =\ \{\,p\in\Sing X\ :\ \text{the local fiber }X_s\cap U_p
\text{ is singular}\,\},
\]
a set independent of $s$ for small $s\neq0$, because each local
classifying arc either lies in the local discriminant or meets it
only at the origin.  Put
\[
X_s^{\flat}=W\cup\bigsqcup_{p\notin\Sigma}\Mil_p ,
\]
a compact oriented $6$-manifold with boundary
$\bigsqcup_{p\in\Sigma}L_p$ contained in the smooth locus of $X_s$.
Then for every $\gamma\in H_3(X_s^{\flat},\partial X_s^{\flat};\QQ)$
there is a class $\kappa\in\bigoplus_pH_2(L_p;\QQ)$ with
\[
\kappa\in\ker\Bigl(\bigoplus_pH_2(L_p;\QQ)\to H_2(\Xh;\QQ)\Bigr),
\qquad
\kappa_p\in \Rsm_p\quad(p\notin\Sigma\text{ a cone point}),
\]
whose component at a node $q\notin\Sigma$ is a fixed nonzero multiple
of $(\gamma\cdot\delta_q)$ times the generator, and whose component at
a cone point $p\notin\Sigma$ is the image of $\gamma$ under the
composite
$H_3(X_s^{\flat},\partial X_s^{\flat};\QQ)\to H_3(\Mil_p,L_p;\QQ)
\xrightarrow{\ \partial,\ \sim\ }\Rsm_p$.  Here $\gamma\cdot\delta_q$ is
the Poincar\'e--Lefschetz pairing of $X_s^{\flat}$.  No \emph{local}
condition is imposed on $\kappa_p$ for $p\in\Sigma$; those coordinates
are of course still coupled to the others by the displayed kernel
condition.
\end{lemma}

\begin{proof}
By Lemma~\ref{lem:rigid} every $\mathbb F_1$-cone point lies in
$\Sigma$, so a cone point $p\notin\Sigma$ is a $\dP_6$- or
$\dP_7$-cone point and the local computations of Lemma~\ref{lem:M}
apply to it.  Apply Mayer--Vietoris for the pair
$(X_s^{\flat},\partial X_s^{\flat})$ with $A=W$,
$B=\bigsqcup_{p\notin\Sigma}\Mil_p$ and
$\partial X_s^{\flat}\subset A$, the couple being made excisive by the
standard collar thickening:
\[
H_3(X_s^{\flat},\partial X_s^{\flat};\QQ)
\xrightarrow{\ \partial\ }
\bigoplus_{p\notin\Sigma}H_2(L_p;\QQ)
\longrightarrow
H_2\bigl(W,\partial X_s^{\flat};\QQ\bigr)\oplus
\bigoplus_{p\notin\Sigma}H_2(\Mil_p;\QQ).
\]
The second summand gives $(\partial\gamma)_p\in \Rsm_p$ at the cone
points $p\notin\Sigma$.  For the first, let $x\in H_2(W;\QQ)$ be the
image of $\partial\gamma$.  It dies in
$H_2(W,\partial X_s^{\flat};\QQ)$, so by exactness of
\[
H_2(\partial X_s^{\flat};\QQ)\longrightarrow H_2(W;\QQ)
\longrightarrow H_2(W,\partial X_s^{\flat};\QQ)
\]
it is the image of some
$y\in H_2(\partial X_s^{\flat};\QQ)=
\bigoplus_{p\in\Sigma}H_2(L_p;\QQ)$.  Set
$\kappa_p=(\partial\gamma)_p$ for $p\notin\Sigma$ and
$\kappa_p=-y_p$ for $p\in\Sigma$.  Then $\kappa$ dies in $H_2(W;\QQ)$,
hence in $H_2(\Xh;\QQ)$ because $W\subset\Xh$.

For the local statements, fix a node $q\notin\Sigma$.  Since
$\partial X_s^{\flat}\subset W$ is disjoint from $\Mil_q$, excision gives
a restriction
$H_3(X_s^{\flat},\partial X_s^{\flat};\QQ)\to
H_3(\Mil_q,L_q;\QQ)$, and under Poincar\'e--Lefschetz duality on
$X_s^{\flat}$ it is dual to the inclusion
$H_3(\Mil_q;\QQ)\to H_3(X_s^{\flat};\QQ)$.  With $F\cdot\delta_q=1$ as in
the proof of Lemma~\ref{lem:M}, the restriction of $\gamma$ is
therefore $(\gamma\cdot\delta_q)F$, and $\kappa_q$ is
$(\gamma\cdot\delta_q)\,\partial F$.  The same excision at a cone
point $p\notin\Sigma$, followed by the isomorphism
$\partial\colon H_3(\Mil_p,L_p;\QQ)\to \Rsm_p$ of Lemma~\ref{lem:M},
gives the second statement.
\end{proof}

\begin{remark}\label{rem:partialform}
The construction of $\Omega$ in Section~\ref{sec:relform} is unchanged
in this setting.  It uses the decomposition
\[
X_s\ =\ W\ \cup\ \bigsqcup_{p\notin\Sigma}\Mil_p\ \cup\
\bigsqcup_{p\in\Sigma}\bigl(X_s\cap U_p\bigr)
\]
along all the links, the vanishing $H^1(L_p;\ZZ)=0$ of
Section~\ref{sec:relform}, which covers the $\mathbb F_1$ link, and the
triviality of $\omega_{\cX/\Disk}$ near each $(p,0)$.  The fiber $X_s$
is allowed to be singular.  That it is Gorenstein needs an argument
now that it is not smooth, and the following one replaces the appeal
to base change: $X_0$ is Gorenstein and is the divisor $\{t=0\}$ of
$\cX$ with $t$ a nonzerodivisor by flatness, so $\cX$ is Gorenstein
near $X_0$; each $X_s$ is then a Cartier divisor in a Gorenstein space
and hence Gorenstein, and adjunction makes $\omega_{\cX/\Disk}$
invertible with $\omega_{\cX/\Disk}|_{X_s}\cong\omega_{X_s}$.  The
singularities of $X_s$ are isolated for small $s$, since the singular
locus of $\cX\to\Disk$ is closed, proper over $\Disk$, and meets $X_0$
in a finite set; the same argument places $\Sing X_s$ inside
$\bigsqcup_p\operatorname{int}U_p$, which is what makes
$X_s^{\flat}$ lie in the smooth locus.  The holomorphic $3$-form restricts
to $X_s^{\flat}$, and the cycles used in Lemmas~\ref{lem:P} and~\ref{lem:D}
lie there by construction.  A cycle that bounds a $4$-chain in
$X_s^{\flat}$ has zero period by Stokes' theorem.  The same period arguments
therefore prove nonvanishing in $H_3(X_s^{\flat};\QQ)$, the group used in
Section~\ref{sec:partner}.  Poincar\'e--Lefschetz duality on the
compact oriented manifold with boundary $X_s^{\flat}$, which pairs
$H_3(X_s^{\flat};\QQ)$ perfectly with
$H_3(X_s^{\flat},\partial X_s^{\flat};\QQ)$, replaces Poincar\'e
duality on $X_s$.
\end{remark}

\section{The distinguished hypersurface and its finite data}
\label{sec:example}

\subsection{The threefold}\label{sec:X0}
Let $\Delta=\Delta_{\mathbb F_1}$ be the $22$-vertex reflexive
$4$-polytope of \cite[Theorem~6.1]{PaperIII} and
$\Delta^{\vee}$ its polar, with the $26$ vertices
$q_1,\dots,q_{26}$ listed in Table~\ref{tab:rays}.  For a reflexive
polytope $\Pi$ we write $\PP_{\Pi}$ for the Gorenstein Fano toric fourfold
whose fan is the face fan of $\Pi$; its anticanonical polytope is the polar
of $\Pi$, so $\PP_{\Delta^{\vee}}$ is the variety written $\PP_{\Delta}$ in
Batyrev's notation~\cite{Batyrev94}.  A Batyrev-regular
anticanonical hypersurface
$X^{\circ}\subset\PP_{\Delta^{\vee}}$ has exactly $26$ ordinary double
points (one for each unit-square $2$-face of $\Delta^\vee$), two
anticanonical $\dP_7$-cone points (the two pentagon faces with one
interior point) and two $\dP_6$-cone points (the two hexagon faces).
This is stated for a generic equation in \cite[Section~6]{PaperIII}; the
following argument applies to every Batyrev-regular equation.  All thirty
germs are \emph{analytically}
the cones: the ambient chart at the $3$-dimensional cone over a
$2$-face $F$ splits as $C(F)\times\CC^{*}$, the restriction of the
defining section to the one-dimensional orbit is
$t^{k}(c_a+c_bt)$ with $c_a,c_b\neq0$ by Batyrev regularity because
the dual edge $F^{*}$ has lattice length one, and Weierstrass
preparation in the orbit coordinate exhibits the germ as a graph over
$(C(F),0)$.  In particular each face contributes exactly one singular
point, the simple zero of $c_a+c_bt$.

$X^{\circ}$ satisfies the standing hypotheses of
Section~\ref{sec:periods}: it is an anticanonical Cartier divisor in
the Gorenstein Fano toric fourfold $P=\PP_{\Delta^{\vee}}$, so
$\omega_{X^{\circ}}\cong\cO_{X^{\circ}}$ by adjunction, and
$h^1(\cO_{X^{\circ}})=0$ from the sequence
$0\to\cO_P(K_P)\to\cO_P\to\cO_{X^{\circ}}\to0$ together with
$H^1(\cO_P)=0$ and
$H^2(\cO_P(K_P))\cong H^{2}(\cO_P)^{\vee}=0$ (toric vanishing and
Serre duality on the complete toric fourfold).

\begin{table}\centering
\caption{The $30$ lattice points of $\partial\Delta^{\vee}$: the
vertices $q_1,\dots,q_{26}$ and the four face-interior points:
$q_{27},q_{28}$ are the interior points of the two hexagons (below
$E_{6,1},E_{6,2}$) and $q_{29},q_{30}$ of the two pentagons (below
$E_{7,1},E_{7,2}$).}\label{tab:rays}
\begin{tabular}{llll}
\toprule
$q_1=(-1,-1,-1,-1)$ & $q_2=(-1,-1,-1,0)$ & $q_3=(-1,-1,0,-1)$ &
$q_4=(-1,-1,0,1)$\\
$q_5=(-1,-1,1,0)$ & $q_6=(-1,-1,1,1)$ & $q_7=(-1,0,-1,-1)$ &
$q_8=(-1,0,-1,0)$\\
$q_9=(-1,0,0,-1)$ & $q_{10}=(-1,0,0,1)$ & $q_{11}=(-1,0,1,0)$ &
$q_{12}=(-1,0,1,1)$\\
$q_{13}=(0,-1,-1,-1)$ & $q_{14}=(0,-1,-1,0)$ & $q_{15}=(0,-1,0,1)$ &
$q_{16}=(0,-1,0,0)$\\
$q_{17}=(0,1,-1,-1)$ & $q_{18}=(0,1,-1,0)$ & $q_{19}=(0,1,0,-1)$ &
$q_{20}=(0,1,0,0)$\\
$q_{21}=(1,1,-1,-1)$ & $q_{22}=(1,0,-1,0)$ & $q_{23}=(0,0,0,-1)$ &
$q_{24}=(0,0,0,1)$\\
$q_{25}=(0,0,1,0)$ & $q_{26}=(0,0,1,1)$ & $q_{27}=(-1,-1,0,0)$ &
$q_{28}=(-1,0,0,0)$\\
$q_{29}=(0,0,-1,-1)$ & $q_{30}=(0,0,-1,0)$ &&\\
\bottomrule
\end{tabular}
\end{table}

\subsection{The resolution and the matrix}\label{sec:matrix}
Five terms recur in what follows, and we fix them here.

A \emph{branch} at a singular point is a reduced irreducible component of the
germ's miniversal base: a node has one, a $\dP_7$ point one (the reduced line of
Section~\ref{sec:miniversal}), and a $\dP_6$ point two, the $A_1$ line and the
$A_2$ plane.  A one-parameter smoothing of $X$ induces at each point an arc into
one branch, and $\Rsm_p$ depends on which.  A \emph{profile} for $X$ is a choice
of branch at every $\dP_6$ point at once; the other germ types offer no choice,
so $X$ with $k$ del Pezzo sextic points has $2^{k}$ profiles, and a criterion
that must hold for every smoothing is evaluated on each.

A coordinate of $\bigoplus_pH_2(L_p;\QQ)$ is \emph{covered} when
Theorem~\ref{thm:necessity} asserts nonvanishing there: every node coordinate,
the $(-2)$-line at every $\dP_7$ point, and the $A_1$ summand at a $\dP_6$ point
on a profile that sends it to the line branch.  The $A_2$ plane is not covered,
whose later treatment is described in Section~\ref{sec:open}.

The \emph{root subspace} of a cone point is $\Rsm_p$ read in $K^{\perp}$ through
Proposition~\ref{prop:kernels}: the $(-2)$-line at a $\dP_7$ point, and the
$A_1$ or the $A_2$ summand of $K^{\perp}\cong A_1\oplus A_2$ at a $\dP_6$ point.
The name records that at degree $6$ these are the root sublattices of
$K^{\perp}$.  At degree $7$ the lattice $K^{\perp}$ has discriminant $7$ and is
not generated by its roots, and the subspace is the line spanned by its unique
$(-2)$-class up to sign.

We call the rows of $B$ at a germ its \emph{relation rows};
Proposition~\ref{prop:T1channels} shows that their complexifications form a
basis of the germ's $T^1$.

Fix the crepant resolution $\Xh\to X^{\circ}$ obtained from the
projective simplicial subdivision of the face fan of $\Delta^\vee$
using all thirty boundary lattice points, with the star subdivision at
the four face-interior points and the twenty-six square diagonals
\[
\begin{aligned}
&(q_2,q_7),(q_2,q_{13}),(q_3,q_7),(q_4,q_8),(q_4,q_{14}),
(q_5,q_9),(q_5,q_{13}),(q_5,q_{23}),(q_9,q_{23}),\\
&(q_6,q_{10}),(q_{10},q_{15}),(q_6,q_{11}),(q_6,q_{16}),
(q_6,q_{25}),(q_8,q_{17}),(q_9,q_{17}),(q_{12},q_{18}),
(q_{12},q_{24}),\\
&(q_{12},q_{19}),(q_{12},q_{25}),(q_{14},q_{16}),(q_{16},q_{23}),
(q_{16},q_{26}),(q_{18},q_{19}),(q_{20},q_{24}),(q_{20},q_{25}),
\end{aligned}
\]
labeled $N_1,\dots,N_{26}$ in this order; the subdivision is smooth
and projective.  Over the nodes $\pi$ is small with
exceptional curves $\Gamma_i$, and over the cone points the
exceptional divisors are toric surfaces
$E_{6,1},E_{6,2}\cong\dP_6$ and $E_{7,1},E_{7,2}\cong\dP_7$.

Write $\widehat\Sigma$ for that subdivided fan and $\PP_{\widehat\Sigma}=\Xh$'s
ambient fourfold $\TV(\widehat\Sigma)$; the letter $Y$ stays with the cone germ
of Section~\ref{sec:germs}.  We record the two families of functionals on
$\Pic(\PP_{\widehat\Sigma})$
(equivalently, by restriction, on $\Pic(\Xh)\otimes\QQ$, though only
one containment is used below) whose kernel computation is the finite
input.  \emph{Node rows}: for each square with cyclic vertices
$(a,b,c,d)$ and chosen diagonal $(a,c)$, the intersection functional
of $\Gamma_i$, with values $-1$ at $D_a,D_c$, $+1$ at $D_b,D_d$ and
$0$ otherwise; it annihilates principal divisors because
$q_a+q_c=q_b+q_d$.  \emph{Divisorial rows}: for each exceptional
surface, the composition of the restriction
$\Pic(\PP_{\widehat\Sigma})\to\Pic(E)$ ($D_j\mapsto$ the boundary curve of $E$ for the
rays of the face, $D_{\mathrm{int}}\mapsto K_E$, all other rays
$\mapsto0$) with a basis of the rank-$2$ or $3$ lattice
$K^{\perp}\subset\Pic(E)^{\vee}$.  In the bases of
Table~\ref{tab:kperp} the three $\dP_6$ functionals
$\varphi_{s_1},\varphi_{s_2},\varphi_{s_3}$ satisfy
$\varphi_{s_i}^2=-2$, $\varphi_{s_1}\cdot\varphi_{s_2}=1$,
$\varphi_{s_3}\perp\varphi_{s_1},\varphi_{s_2}$, so
$\langle\varphi_{s_3}\rangle$ is the $A_1$ summand and
$\langle\varphi_{s_1},\varphi_{s_2}\rangle$ the $A_2$ summand; the
two $\dP_7$ functionals satisfy $\varphi_\beta^2=-2$,
$\varphi_\alpha^2=-4$, and $\varphi_\beta$ spans the unique
$(-2)$-line.  Thus, if $e_{p,i}$ is a chosen basis class of $H_2(L_p;\QQ)$
and $\iota_p\colon L_p\hookrightarrow\Xh$ is the inclusion, our convention is
\[
B_{(p,i),j}=D_j\cdot(\iota_p)_*e_{p,i}.
\]
Rows are indexed by local homology classes and columns by ambient divisor
classes; a relation with coefficient vector $\kappa$ satisfies
$B^{\mathsf T}\kappa=0$.
Stacking the $26$ node rows and the ten divisorial rows
gives an integral matrix
$B\in\operatorname{Mat}_{36\times26}(\ZZ)$ in the dual basis of the
$26$ divisor classes complementary to the unimodular cone
$(q_{15},q_{22},q_{24},q_{26})$.

\begin{table}\centering
\caption{Exceptional-surface data.  Bases of $\Pic(E)$: boundary
divisors at the rays listed; intersection matrices $I$; canonical
classes $K$; and the $K^{\perp}$ functionals in the dual bases.  The
two $\dP_6$ surfaces carry identical data at rays
$(q_3,q_4,q_5,q_6)$, resp.\ $(q_9,q_{10},q_{11},q_{12})$.}
\label{tab:kperp}
\begin{tabular}{lll}
\toprule
surface & lattice data & $K^\perp$ rows\\
\midrule
$E_{6,i}$ &
$I_6=\begin{psmallmatrix*}[r]-1&0&1&0\\0&-1&0&1\\1&0&-1&1\\0&1&1&-1
\end{psmallmatrix*}$,\quad $K=(-1,-1,-2,-2)$ &
$\begin{aligned}\varphi_{s_1}&=(0,0,-1,1)\\
\varphi_{s_2}&=(-1,-1,1,0)\\
\varphi_{s_3}&=(-1,1,1,-1)\end{aligned}$\\[1.4em]
$E_{7,1}$ &
$I_{7,1}=\begin{psmallmatrix*}[r]0&0&1\\0&-1&1\\1&1&0
\end{psmallmatrix*}$ at $(q_{13},q_{17},q_{21})$,\quad
$K=(-1,-1,-2)$ &
$\begin{aligned}\varphi_{\alpha}&=(1,-1,0)\\
\varphi_{\beta}&=(1,1,-1)\end{aligned}$\\[1.2em]
$E_{7,2}$ &
$I_{7,2}=\begin{psmallmatrix*}[r]-1&0&1\\0&0&1\\1&1&0
\end{psmallmatrix*}$ at $(q_{14},q_{18},q_{22})$,\quad
$K=(-1,-1,-2)$ &
$\begin{aligned}\varphi_{\alpha}&=(-2,0,1)\\
\varphi_{\beta}&=(-1,-1,1)\end{aligned}$\\
\bottomrule
\end{tabular}
\end{table}

\subsection{The ambient kernel computation}\label{sec:kernelcomp}
Write coordinates on $\QQ^{36}$ as
\[
(u_1,\dots,u_{26};\ s_{1},s_{2},s_{3}\text{ at }E_{6,1};\
s_{1},s_{2},s_{3}\text{ at }E_{6,2};\ \alpha_1,\beta_1;\
\alpha_2,\beta_2).
\]
The labels $s_1,s_2,s_3$ are chosen to match those of
Section~\ref{sec:miniversal}: there $\{s_1=s_2=0\}$ cuts out the line
component of the miniversal base and $\{s_3=0\}$ the plane component, and here
the same two equations cut out the $A_1$ line and the $A_2$ plane of
$K^{\perp}$.  These are different spaces and the coincidence of equations is a
choice of labelling, not an identification.
In these fixed coordinates, every $\kappa\in\ker(B^{\mathsf T})$ satisfies
\[
u_9=u_{11}=\alpha_1-\alpha_2,\qquad
\beta_1=0,\qquad \beta_2=-\alpha_2.
\]
Indeed, summing the divisor-pairing equations for
$D_1,\dots,D_{12},D_{17},\dots,D_{20}$ gives the formula for $u_9$;
negating the sum for $D_7,\dots,D_{12},D_{17},\dots,D_{21}$ gives
the formula for $u_{11}$.  The negative of the $D_{21}$ equation gives
$\beta_1=0$, and the sum for $D_1,\dots,D_{12}$ minus the $D_{21}$
equation gives $\beta_2=-\alpha_2$.
The Milnor-fiber conditions at the two degree-$7$ points impose
$\alpha_1=\alpha_2=0$, so these identities force
$u_9=u_{11}=\beta_1=\beta_2=0$, independently of the degree-$6$ profile.
In particular, they exclude the nonzero nodal coefficients required by
Theorem~\ref{thm:necessity}.
The following further statements are exact integral computations from the data of
Tables~\ref{tab:rays} and~\ref{tab:kperp} and the subdivision of
Section~\ref{sec:matrix}:
\begin{enumerate}
\item $\rk B=21$ and $\dim\ker(B^{\mathsf T})=15$; the node
subsystem has rank $18$, with coloops exactly $N_9$ and $N_{11}$; the
unique row coloop of the full matrix is the $\beta_1$-row, i.e.\
$\beta_1$ vanishes identically on $\ker(B^{\mathsf T})$.  This last
statement is read in the basis of Table~\ref{tab:kperp}: replacing
$\varphi_\alpha$ by $\varphi_\alpha+c\varphi_\beta$ shears the kernel
coordinates and can destroy it.  What is basis-free is the restricted
statement of item~(2), since once $\alpha=0$ the coordinate $\beta$ does
not depend on the choice of complement to the $(-2)$-line.  The second
proof of Theorem~\ref{thm:main} uses item~(1) in this fixed basis;
item~(2) gives the same conclusion in basis-free form.
\item On the \emph{enlarged restricted kernel}
$\ker(B^{\mathsf T})\cap\{\alpha_1=\alpha_2=0\}$, of dimension $13$,
the four coordinates $u_9,u_{11},\beta_1,\beta_2$ vanish identically.
The four finer profiles obtained by additionally restricting each
$\dP_6$ block to its $A_1$ line ($s_1=s_2=0$) or its $A_2$ plane
($s_3=0$) have dimensions $9,10,10,12$ and force the same four
coordinates, together with one $s_3$ on each mixed profile.  The last dimension
is $12$ and not $11$ because the two $s_3$-coordinates are proportional on the
enlarged restricted kernel: imposing either alone drops the dimension by one,
and imposing the second drops it no further.  That proportionality is also why
a mixed profile, which deletes one $s_3$, forces the other.
\item The vanishing conclusion of item~(2) fails if either restriction
$\alpha_i=0$ is removed.
\end{enumerate}
By construction $\ker(B^{\mathsf T})$ contains the kernel of
$\bigoplus_pH_2(L_p;\QQ)\to H_2(\Xh;\QQ)$ written in these
coordinates: the node generators map to $\pm[\Gamma_i]$ and the
divisorial classes to their pushforwards from
$K^{\perp}\subset H_2(E_p)$, and the entries of each row are the pairings
of its local homology class with the ambient divisor classes.
Only this containment is used: if the ambient
Picard group under-detected $H_2(\Xh;\QQ)$, the ambient kernel
would be larger than the topological one and every forced-zero
statement would still bind.  Equality in fact holds here
($h^{1,1}(\Xh)=26$, the thirty boundary lattice points of $\Delta^{\vee}$ less
four),
and holds for every
polytope in the class considered in this paper, by
Proposition~\ref{prop:kernelequal}.

\section{Proofs of the main results}\label{sec:proofs}

\begin{theorem}[mixed necessity]\label{thm:necessity}
Let $X$ be a connected compact complex threefold with $\omega_X\cong\cO_X$ and
$h^1(\cO_X)=0$ whose singular points are ordinary double points and
anticanonical cones over $\dP_6$ or $\dP_7$, let
$\Xh\to X$ be a crepant resolution as in
Section~\ref{sec:intro}, and suppose $X$ admits a one-parameter
smoothing.  With $\Rsm_p=\ker(H_2(L_p;\QQ)\to H_2(\Mil_p;\QQ))$ for the
actual Milnor fibers of the induced local smoothings, set
\[
K(X)=\Bigl\{\kappa\in\ker\Bigl(\bigoplus_pH_2(L_p;\QQ)
\to H_2(\Xh;\QQ)\Bigr)\ :\ \kappa_p\in \Rsm_p
\text{ at every cone point }p\Bigr\}.
\]
Then:
\begin{enumerate}
\item for every node $q$ there is $\kappa\in K(X)$ with
$\kappa_q\neq0$;
\item for every $\dP_7$ point $q$ there is $\kappa\in K(X)$ whose
$q$-component is a nonzero element of the $(-2)$-line $\Rsm_q$;
\item for every $\dP_6$ point $q$ smoothed through the one-dimensional
component of its miniversal base there is $\kappa\in K(X)$ whose
$q$-component is a nonzero element of $\Rsm_q$ (the $A_1$ summand).
\end{enumerate}
\end{theorem}

\begin{proof}
Fix small $s\neq0$.  (1)~By Lemma~\ref{lem:P}, $\delta_q\neq0$ in
$H_3(X_s;\QQ)$; by Poincar\'e duality on the closed oriented $X_s$
choose $\gamma$ with $\gamma\cdot\delta_q\neq0$; by
Lemma~\ref{lem:M}, $\kappa=\partial\gamma$ lies in $K(X)$ and its
$q$-component is a nonzero multiple of $\gamma\cdot\delta_q$.
(2),(3)~By Lemma~\ref{lem:D} the generator $\nu$ of
$H_3(\Mil_q;\QQ)$ survives in $H_3(X_s;\QQ)$; choose $\gamma$ with
$\gamma\cdot\nu_*\neq0$; by Lemma~\ref{lem:M} the $q$-component of
$\partial\gamma$ is the image of a nonzero element of
$H_3(\Mil_q,L_q;\QQ)$ under the isomorphism onto $\Rsm_q$, and $\Rsm_q$ is
identified with the stated root subspace by
Proposition~\ref{prop:kernels}, Section~\ref{sec:match} and
Lemma~\ref{lem:boundary}.
\end{proof}

\begin{remark}\label{rem:single}
Since a $\QQ$-vector space is not the union of finitely many proper
subspaces, the per-point conclusions of
Theorem~\ref{thm:necessity} combine: there is a \emph{single} class
$\kappa\in K(X)$ simultaneously nonzero at every node, every
$\dP_7$ point, and every line-smoothed $\dP_6$ point.
\end{remark}

\begin{corollary}\label{cor:friedman}
In the purely nodal case, Theorem~\ref{thm:necessity}(1) with
Remark~\ref{rem:single} states: any one-parameter smoothing forces a
relation among the classes $[\Gamma_i]\in H_2(\Xh;\QQ)$ of the
exceptional curves of a small resolution in which every node appears
with nonzero coefficient.  \qed
\end{corollary}

\begin{proposition}\label{prop:coloops}
If $X^{\circ}$ admitted a one-parameter smoothing, the vanishing
spheres of the nodes $N_9$ and $N_{11}$ would satisfy
$\delta_9=\delta_{11}=0$ in $H_3(X_s;\QQ)$.
\end{proposition}

\begin{proof}
By Lemma~\ref{lem:M} every boundary class
$\partial\gamma$ lies in $K(X^{\circ})$; its $\dP_7$-components lie
on the $(-2)$-lines by Proposition~\ref{prop:kernels}, and
Lemma~\ref{lem:boundary} identifies these with the $(-2)$-lines of
the resolution-side coordinates of Section~\ref{sec:kernelcomp}, so
$\partial\gamma$ lies in the enlarged restricted kernel, on which
$u_9$ and $u_{11}$ vanish identically (item~(2) there).  Since the $u_q$-component of $\partial\gamma$ is a
nonzero multiple of $\gamma\cdot\delta_q$, nondegeneracy of the
intersection pairing gives $\delta_9=\delta_{11}=0$.
\end{proof}

\begin{theorem}\label{thm:main}
$X^{\circ}$ admits no one-parameter smoothing.
\end{theorem}

\begin{proof}
Suppose one exists.  Every germ of $X^{\circ}$ is then smoothed; in
particular the node $N_9$ is, so Lemma~\ref{lem:P} gives
$\delta_9\neq0$ in $H_3(X_s;\QQ)$, contradicting
Proposition~\ref{prop:coloops}.

Alternatively, avoiding the nodes altogether: the point covered by
$E_{7,1}$ is smoothed through the unique component of its miniversal
base, so Theorem~\ref{thm:necessity}(2) produces
$\kappa\in K(X^{\circ})$ whose $E_{7,1}$-component is a nonzero
element of the $(-2)$-line, i.e.\ (Lemma~\ref{lem:boundary} and
Table~\ref{tab:kperp}) has nonzero $\beta_1$-coordinate; but
$\beta_1$ vanishes identically already on the unrestricted
$\ker(B^{\mathsf T})$ (Section~\ref{sec:kernelcomp}, item~(1)).
\end{proof}

\begin{remark}\label{rem:independence}
The two proofs are independent in the period lemma and in the
kernel coordinate used (Lemma~\ref{lem:P} with the vanishing of $u_9$ in
Section~\ref{sec:kernelcomp}(2), versus Lemma~\ref{lem:D} with the vanishing
of $\beta_1$ in Section~\ref{sec:kernelcomp}(1), which item~(2) restates in
basis-free form), and share the membership and
boundary-compatibility inputs (Lemmas~\ref{lem:M}
and~\ref{lem:boundary}) and the divisorial rows of
Table~\ref{tab:kperp}.
\end{remark}

\begin{corollary}\label{cor:def}
Let $X$ satisfy the standing hypotheses of Theorem~\ref{thm:necessity} and
admit no one-parameter smoothing.  Then no fiber of a small representative
of the miniversal deformation of $X$ is smooth; consequently $X$ is not
smoothable over any base germ, and neither is any small deformation of $X$.
Likewise, no formal arc in the miniversal deformation of $X$ meets all the
local smoothing loci.  Here a formal arc $\operatorname{Spf}\CC[[t]]\to\Def(X)$
\emph{meets the smoothing locus} of a singular point $p$ if the zeroth Fitting
ideal of the coordinate ring of the relative critical locus of the induced
deformation of $(X,p)$ over $\CC[[t]]$, a finite $\CC[[t]]$-module, is nonzero, that is, if that
deformation has smooth generic fiber.  In particular all of this holds for
$X^{\circ}$, by Theorem~\ref{thm:main}.
\end{corollary}

\begin{proof}
The miniversal deformation of the compact $X$ exists
\cite{Grauert74} and its discriminant is a closed analytic subset of
the base.  If some fiber arbitrarily close to the origin were smooth,
the curve selection lemma would produce a
real-analytic arc through the origin with smooth generic fiber
(the complement of the discriminant is semianalytic, so by
\cite[\S19, Prop.~2]{Lojasiewicz65} it contains a semianalytic arc
ending at the origin, which by \cite[Lemma~6.3(1)]{BM88} is the image
of $(0,\varepsilon)$ under a real-analytic map on
$(-\varepsilon,\varepsilon)$; for real algebraic sets see
\cite{Milnor68});
complexifying (a real-analytic function vanishing on a real segment
of a disk vanishes identically, and the discriminant preimage of the
complexified arc is discrete) gives a holomorphic disk, that is, a
one-parameter smoothing, contrary to hypothesis.  Small
deformations are covered by openness of versality.  For the formal
statement, suppose a formal arc meets every local smoothing locus, and let
$k$ be the largest $t$-adic order of the zeroth Fitting ideals in the
definition.  By Artin approximation \cite{Artin68} applied to the
analytic miniversal deformation, some analytic arc agrees with the formal
one modulo $t^{k+1}$.  Fitting ideals commute with base change, so the
corresponding Fitting ideals of the analytic arc agree with those of the
formal arc modulo $t^{k+1}$ and are therefore nonzero: for small $t\neq0$
the fibers of the analytic arc are smooth near every singular point
of $X$.  Since the singular locus of the
fibers is closed in the total space and meets the central fiber only
in $\Sing X$, these fibers are, for
small $t$, singular at most inside the Milnor-ball
neighborhoods of those points, hence smooth everywhere, so the arc is a
one-parameter smoothing, again contrary to hypothesis.
\end{proof}

\begin{remark}\label{rem:mirror}
The two nodal coloops quantify how far $X^{\circ}$ is from the nodal
theory: no relation among the twenty-six exceptional curves alone is
nonzero at $N_9$ or $N_{11}$, so a smoothing relation would have to include
classes from the exceptional divisors.  Items~(2)--(3) of
Section~\ref{sec:kernelcomp} show that mixed relations satisfying the local
Milnor-fiber conditions still have zero coefficients at those nodes.
The mirror partner of
$X^{\circ}$ is treated in Section~\ref{sec:partner}.
\end{remark}

\begin{table}\centering
\caption{The $24$ boundary lattice points of $\Delta=\Delta_{\mathbb
F_1}$: the vertices $v_1,\dots,v_{22}$ in the order of
\cite[Theorem~6.1]{PaperIII}, the interior point $v_{23}$ of the
hexagonal face and the interior point $v_{24}$ of the $\mathbb F_1$
face.}\label{tab:partnerrays}
\begin{tabular}{llll}
\toprule
$v_1=(1,0,0,0)$ & $v_2=(0,1,0,0)$ & $v_3=(1,-1,0,0)$ &
$v_4=(0,0,1,0)$\\
$v_5=(0,0,0,1)$ & $v_6=(0,0,1,-1)$ & $v_7=(0,0,-1,1)$ &
$v_8=(0,0,0,-1)$\\
$v_9=(0,0,-1,0)$ & $v_{10}=(-1,1,0,0)$ & $v_{11}=(0,-1,0,0)$ &
$v_{12}=(-1,1,-1,1)$\\
$v_{13}=(0,-1,0,-1)$ & $v_{14}=(0,-1,-1,0)$ & $v_{15}=(-1,1,-1,0)$ &
$v_{16}=(-1,0,0,-1)$\\
$v_{17}=(-1,0,-1,1)$ & $v_{18}=(-1,-1,0,-1)$ & $v_{19}=(-2,1,-1,1)$ &
$v_{20}=(-2,1,-1,0)$\\
$v_{21}=(-1,-1,-1,0)$ & $v_{22}=(-2,0,-1,0)$ & $v_{23}=(-1,0,-1,0)$ &
$v_{24}=(-1,0,0,0)$\\
\bottomrule
\end{tabular}
\end{table}

\begin{table}\centering
\caption{The two exceptional surfaces of $\Xh_{\Delta}$: the cyclic
boundary rays, the intersection matrix $I$ in the Picard basis of the
boundary divisors at the rays listed, the canonical class $K$, and the
$K^{\perp}$ functionals in the dual basis.  As in Table~\ref{tab:kperp},
$\varphi_{s_3}$ spans the $A_1$ summand and
$\varphi_{s_1},\varphi_{s_2}$ the $A_2$ summand, here with
$\varphi_{s_1}\cdot\varphi_{s_2}=-1$.  The $\mathbb F_1$ lattice
$K^{\perp}$ has no $(-2)$-class.}\label{tab:partnersurf}
\begin{tabular}{lll}
\toprule
surface & lattice data & $K^{\perp}$ rows\\
\midrule
$E_{6,\Delta}$ over $v_{23}$, rays &
$I_6=\begin{psmallmatrix*}[r]-1&1&0&0\\1&-1&1&0\\0&1&-1&1\\
0&0&1&-1\end{psmallmatrix*}$ at $(v_9,v_{15},v_{20},v_{22})$, &
$\begin{aligned}\varphi_{s_1}&=(1,-1,0,1)\\
\varphi_{s_2}&=(0,-1,1,0)\\
\varphi_{s_3}&=(1,-1,1,-1)\end{aligned}$\\
$(v_{21},v_{14},v_9,v_{15},v_{20},v_{22})$ &
$K=(-1,-2,-2,-1)$ &\\[1.2em]
$E_{\mathbb F_1}$ over $v_{24}$, rays &
$I_{\mathbb F_1}=\begin{psmallmatrix*}[r]0&1\\1&-1
\end{psmallmatrix*}$ at $(v_{19},v_{22})$, &
$\varphi_{t}=(-2,3)$\\
$(v_{18},v_4,v_{19},v_{22})$ & $K=(-3,-2)$ &\\
\bottomrule
\end{tabular}
\end{table}

\section{The mirror partner}\label{sec:partner}
The polytope $\Delta$ itself carries a Calabi--Yau hypersurface, the
mirror partner of $X^{\circ}$ in the doubly isolated pair.  Write
$X_{\Delta}\subset\PP_{\Delta}$ for a Batyrev-regular anticanonical
hypersurface, with the convention of Section~\ref{sec:X0} ($\PP_{\Delta}$ has
the face fan of $\Delta$ as its fan and $\Delta^{\vee}$ as its anticanonical
polytope), and $v_1,\dots,v_{22}$ for the vertices of $\Delta$ in
the order of \cite[Theorem~6.1]{PaperIII}, reproduced in
Table~\ref{tab:partnerrays}.  By that theorem $X_{\Delta}$ has twenty
ordinary double points, one $\dP_6$-cone point over the hexagonal face
$\conv\{v_9,v_{14},v_{15},v_{20},v_{21},v_{22}\}$, and one point whose
germ is the anticanonical cone over $\mathbb F_1$, over the
quadrilateral face $\conv\{v_4,v_{18},v_{19},v_{22}\}$.  As in
Section~\ref{sec:X0}, $X_{\Delta}$ is an anticanonical Cartier divisor
in a Gorenstein Fano toric fourfold, so
$\omega_{X_{\Delta}}\cong\cO_{X_{\Delta}}$ and
$h^1(\cO_{X_{\Delta}})=0$, and every germ is analytically the cone
over its face because every dual edge has lattice length one.  The
last germ
is reduced-rigid, so $X_{\Delta}$ admits no smoothing \cite{PaperI}.
The reduced-rigid $\mathbb F_1$ germ already prevents a smoothing of
$X_{\Delta}$.  Independently, Theorem~\ref{thm:partial} shows that the global
obstruction is also present on $X_{\Delta}$ and does not use that germ.

\begin{theorem}[necessity with germs left singular]\label{thm:partial}
Let $X$ be a connected compact complex threefold with $\omega_X\cong\cO_X$ and
$h^1(\cO_X)=0$ whose singular points are ordinary double points,
anticanonical cones over $\dP_6$ or $\dP_7$, and anticanonical cones
over $\mathbb F_1$; let $\Xh\to X$ be a crepant resolution as in
Section~\ref{sec:intro}.  Let $\cX\to\Disk$ be any one-parameter
deformation of $X$ and let $\Sigma$ be the set of points at which the
nearby fiber is singular, a set containing every $\mathbb F_1$-cone
point by Lemma~\ref{lem:rigid}.  Put
\[
K^{\flat}(X;\Sigma)=
\Bigl\{\kappa\in\ker\Bigl(\bigoplus_pH_2(L_p;\QQ)\to H_2(\Xh;\QQ)
\Bigr)\ :\ \kappa_p\in \Rsm_p\text{ at every cone point }p\notin\Sigma
\Bigr\}.
\]
Then $K^{\flat}(X;\Sigma)$ contains a class nonzero at $q$ for every
node $q\notin\Sigma$, for every $\dP_7$ point $q\notin\Sigma$, and for
every $\dP_6$ point $q\notin\Sigma$ smoothed through the one-dimensional
component of its miniversal base.  For $\Sigma=\varnothing$ this is
Theorem~\ref{thm:necessity}.
\end{theorem}

\begin{proof}
The proof of Theorem~\ref{thm:necessity} applies with
Remark~\ref{rem:partialform} and Lemma~\ref{lem:Mpartial} in place of
Lemma~\ref{lem:M}.  At a node $q\notin\Sigma$, Lemma~\ref{lem:P}
gives $\delta_q\neq0$ in $H_3(X_s^{\flat};\QQ)$, since a bounding
$4$-chain in $X_s^{\flat}$ would make the period of the closed form
$\Omega_s$ vanish; Poincar\'e--Lefschetz duality on the compact
oriented manifold with boundary $X_s^{\flat}$ produces
$\gamma\in H_3(X_s^{\flat},\partial X_s^{\flat};\QQ)$ with
$\gamma\cdot\delta_q\neq0$; and Lemma~\ref{lem:Mpartial} converts
$\gamma$ into a class $\kappa\in K^{\flat}(X;\Sigma)$ with
$\kappa_q\neq0$.  At the cone points $q\notin\Sigma$ in the two
covered branches the same substitution is made in the argument of
Theorem~\ref{thm:necessity}(2),(3), with Lemma~\ref{lem:D} read in
$H_3(X_s^{\flat};\QQ)$.
\end{proof}

Whatever the deformation and whatever $\Sigma$, the space
$K^{\flat}(X;\Sigma)$ is by definition contained in
$\ker\bigl(\bigoplus_pH_2(L_p;\QQ)\to H_2(\Xh;\QQ)\bigr)$, so a
coordinate that vanishes identically on the latter is forced in every
case.  On $X_{\Delta}$ that is what happens.

Fix the smooth projective crepant subdivision of the face fan of
$\Delta$ that uses all twenty-four boundary lattice points, with the
star subdivision at $v_{23}$ and $v_{24}$ and the twenty square
diagonals
\[
\begin{aligned}
&(v_2,v_7),(v_2,v_9),(v_3,v_8),(v_3,v_9),(v_{10},v_{12}),
(v_{10},v_{15}),(v_{11},v_{13}),(v_{11},v_{14}),(v_{11},v_{17}),\\
&(v_{10},v_{16}),(v_9,v_{12}),(v_{12},v_{17}),(v_{14},v_{17}),
(v_9,v_{13}),(v_{13},v_{16}),(v_{15},v_{16}),(v_{15},v_{19}),
(v_{14},v_{18}),\\
&(v_{18},v_{20}),(v_{19},v_{21}),
\end{aligned}
\]
labeled $n_1,\dots,n_{20}$; the subdivision is smooth, with $94$
maximal cones, and projective.  The corresponding circuit
relations are
\[
\begin{aligned}
&v_1{+}v_{12}=v_2{+}v_7,\ \
v_1{+}v_{15}=v_2{+}v_9,\ \
v_1{+}v_{13}=v_3{+}v_8,\ \
v_1{+}v_{14}=v_3{+}v_9,\\
&v_2{+}v_{19}=v_{10}{+}v_{12},\ \
v_2{+}v_{20}=v_{10}{+}v_{15},\ \
v_3{+}v_{18}=v_{11}{+}v_{13},\ \
v_3{+}v_{21}=v_{11}{+}v_{14},\\
&v_5{+}v_{21}=v_{11}{+}v_{17},\ \
v_6{+}v_{20}=v_{10}{+}v_{16},\ \
v_7{+}v_{15}=v_9{+}v_{12},\ \
v_7{+}v_{19}=v_{12}{+}v_{17},\\
&v_7{+}v_{21}=v_{14}{+}v_{17},\ \
v_8{+}v_{14}=v_9{+}v_{13},\ \
v_8{+}v_{18}=v_{13}{+}v_{16},\ \
v_8{+}v_{20}=v_{15}{+}v_{16},\\
&v_{12}{+}v_{20}=v_{15}{+}v_{19},\ \
v_{13}{+}v_{21}=v_{14}{+}v_{18},\ \
v_{16}{+}v_{22}=v_{18}{+}v_{20},\ \
v_{17}{+}v_{22}=v_{19}{+}v_{21},
\end{aligned}
\]
in the order $n_1,\dots,n_{20}$; in each of them the right-hand pair
is the chosen diagonal, so the circuit row of
Section~\ref{sec:matrix} is $-1$ there and $+1$ on the left-hand pair.
Choosing the other diagonal negates the row and changes nothing
below.
Table~\ref{tab:partnersurf} records the two exceptional surfaces.  The
lattice $K^{\perp}(\mathbb F_1)$ is a line spanned by a primitive
class of square $-8$, and it contains no $(-2)$-class, consistently
with the reduced-rigidity of that germ.  No condition is imposed at
that point in any case: by Lemma~\ref{lem:rigid} it lies in $\Sigma$
for every deformation, so no Milnor fiber and no $\Rsm_p$ is defined
there.  Stacking the twenty node rows, the three hexagon rows and
the single $\mathbb F_1$ row gives an integral
matrix $B_{\Delta}\in\operatorname{Mat}_{24\times24}(\ZZ)$, the
columns being indexed by the divisors of the twenty-four boundary
lattice points, which span $\Pic$ of the resolution without being a
basis.  The
following are exact integral computations, reproduced by the ancillary
programs from the printed data of
Tables~\ref{tab:partnerrays} and~\ref{tab:partnersurf}:
\begin{enumerate}
\item $\rk B_{\Delta}=16$ and $\dim\ker(B_{\Delta}^{\mathsf T})=8$;
the node subsystem has rank $14$, with coloops exactly $n_9$ and
$n_{10}$;
\item on the unrestricted kernel $\ker(B_{\Delta}^{\mathsf T})$ the
coordinates of $n_9$, of $n_{10}$, of the $\mathbb F_1$ row and of the
$A_1$ row of the hexagon vanish identically, so no branch restriction
is needed;
\item the column of $B_{\Delta}$ at the divisor of $v_{24}$ is
identically zero, the node rows having no entry there and the
divisorial rows having entry $\langle\varphi,K_E\rangle=0$; the same
holds at $v_{23}$.  Discarding the $\mathbb F_1$ row and that column
altogether leaves a system on which $n_9$ and $n_{10}$ are still
forced.
\end{enumerate}
\begin{theorem}\label{thm:partner}
No one-parameter deformation of $X_{\Delta}$ has nearby fiber smooth
at the node $n_9=\conv\{v_5,v_{11},v_{17},v_{21}\}$, and none has
nearby fiber smooth at $n_{10}=\conv\{v_6,v_{10},v_{16},v_{20}\}$.
The same two conclusions hold for deformations over arbitrary analytic
base germs.  In particular no analytic deformation of $X_{\Delta}$
smooths all twenty nodes and the $\dP_6$ point, whatever it does at the
$\mathbb F_1$ point.  By item~(2) and
Theorem~\ref{thm:partial}, the same argument applied to the $A_1$
coordinate shows that no analytic deformation of $X_{\Delta}$ smooths
its $\dP_6$ point through the one-dimensional component of its miniversal
base.  The reduced-rigidity of the $\mathbb F_1$ germ excludes the
smoothings of $X_{\Delta}$; the deformations excluded here are ones
that it leaves available.
\end{theorem}

\begin{proof}
Suppose the nearby fiber is smooth at $n_9$, so $n_9\notin\Sigma$.  By
Theorem~\ref{thm:partial} there is
$\kappa\in K^{\flat}(X_{\Delta};\Sigma)$ with $\kappa_{n_9}\neq0$.
But $K^{\flat}(X_{\Delta};\Sigma)$ is contained in
$\ker(B_{\Delta}^{\mathsf T})$ for every $\Sigma$, read in the
coordinates above, and the $n_9$ coordinate vanishes identically
there.  The argument for $n_{10}$ is the same.

For the $\dP_6$ assertion, suppose first that the base is a disk and
that the local classifying arc lies in the one-dimensional component of
the miniversal base.  Theorem~\ref{thm:partial} then gives a class whose
$A_1$ coordinate is nonzero, contrary to item~(2).

Now let the deformation have an arbitrary analytic base germ $(S,0)$.
Since the assertions concern actual nearby fibers, we may replace $S$
by $S_{\mathrm{red}}$.  If one of the excluded local smoothings occurred
arbitrarily close to the origin, the curve-selection argument used in
Corollary~\ref{cor:def}, applied to the complement of the inverse image
of the relevant local discriminant, would give a holomorphic disk
through the origin whose punctured fibers have the same local smoothing
property.  For the $\dP_6$ assertion, apply curve selection inside the
inverse image of the one-dimensional component and outside its
discriminant.  Pullback to the resulting disk contradicts the
one-parameter conclusions just proved.
\end{proof}

\begin{corollary}\label{cor:pair}
Both members of the doubly isolated mirror pair $(X^{\circ},X_{\Delta})$ carry the
obstruction of this paper: on $X^{\circ}$, in the form of
Theorem~\ref{thm:necessity}, it excludes every smoothing; on
$X_{\Delta}$, in the form of Theorem~\ref{thm:partial}, it excludes
every analytic deformation smoothing $n_9$ or $n_{10}$.  On
$X_{\Delta}$ the two mechanisms are independent: by item~(2) the
forcing holds on the unrestricted kernel, where no local condition is imposed
at the $\mathbb F_1$ point, and by item~(3) it persists when the $\mathbb F_1$
data are discarded altogether, whereas the reduced-rigidity of that germ
is a statement about it alone.
\end{corollary}

\section{What the criterion measures}\label{sec:measures}

The rows of Section~\ref{sec:matrix} were introduced as a device: the
$K^{\perp}$ functionals stand in for the exceptional curve class that a
del Pezzo cone point does not have.  This section identifies their span with a
rational form of the germ's $T^1$.  It also makes
the finite test of Section~\ref{sec:test} readable and locates the criterion
against the smoothing theorems of Gross~\cite{Gross97}.

\subsection{The combinatorial setting}\label{sec:combsetting}
Throughout this section and the next, $\Pi\subset N_{\RR}\cong\RR^4$ is a
reflexive $4$-polytope with polar $\Pi^{*}$, $\PP_{\Pi}$ is the Gorenstein
Fano toric fourfold of its face fan, and $X\subset\PP_{\Pi}$ is an
anticanonical hypersurface regular in Batyrev's sense
\cite[Definition~3.1.1]{Batyrev94}.  (The letter $\Delta$ stays reserved
for the fixed $\Delta_{\mathbb F_1}$ of Section~\ref{sec:X0}; the
polytope relevant to $X^{\circ}$ is $\Pi=\Delta^{\vee}$, and the
polytope relevant to $X_{\Delta}$ is $\Pi=\Delta$.)

For a $2$-face $G$ of $\Pi$ with dual edge $G^{*}\subset\Pi^{*}$ of
lattice length $\ell(G^{*})$, the hypersurface meets the closed orbit
of $\Cone(G)$ in $\ell(G^{*})$ points, and at each of them the germ of
$X$ is the affine toric threefold of $\Cone(G)$, exactly as in
Section~\ref{sec:X0} \cite{Batyrev94}, \cite[Proposition~3.1]{PaperI}.  The
germ is singular
unless $G$ is a unimodular triangle, and it is isolated when every edge
of $G$ has lattice length one; we call such a $G$ a \emph{singular
$2$-face} and write $(Y_G,0)$ for the germ.  A unit square gives a node;
a pentagon with unit edges and one interior lattice point gives $C(\dP_7)$
and a hexagon with unit edges and one interior lattice point gives
$C(\dP_6)$, the interior point
being the origin of the fan of the exceptional surface in each case.
(A polygon with exactly one interior lattice point is reflexive, and among
the sixteen reflexive polygons the only ones with five or six vertices and no
further boundary lattice points are the fan polygons of $\dP_7$ and $\dP_6$.)

\begin{definition}\label{def:admissible}
$\Pi$ is \emph{admissible} if every edge of $\Pi$ has lattice length
one; every $2$-face of $\Pi$ has either three vertices and no interior
lattice point, or four vertices and no interior lattice point, or five or
six vertices and exactly one interior lattice point; and $\ell(G^{*})=1$
for every singular $2$-face $G$.
\end{definition}

The four types are the unimodular triangle, the unit square, the pentagon
and the hexagon listed above: with unit edges, a quadrilateral
without interior points has normalized area $2$ by Pick's formula, so
both triangles cut off by a diagonal are unimodular, and convexity then
forces $v_a+v_c=v_b+v_d$: it is a unimodular parallelogram.  Note that a quadrilateral with unit edges
and \emph{one} interior point is a different face, giving a cone point over
a smooth toric del Pezzo surface of degree $8$, either $\mathbb F_1$ or
$\PP^1\times\PP^1$; both are excluded here, and the $\mathbb F_1$ case is the
one that occurs on $\Delta$ and is treated in Section~\ref{sec:partner}.

For admissible $\Pi$ the hypersurface $X$ satisfies the standing
hypotheses of Theorem~\ref{thm:necessity}, by the argument of
Section~\ref{sec:X0}: $\Sing X$ is finite with one point per singular
$2$-face, and every germ is a node or an anticanonical $\dP_6$ or
$\dP_7$ cone.  The crepant resolution is taken throughout from a maximal
projective crepant partial desingularization of $\PP_{\Pi}$
\cite[\S2.2]{Batyrev94} whose fan refines the star subdivision of the cone
over each pentagonal or hexagonal $2$-face at its interior lattice point,
restricted to $X$ \cite[\S4.2]{Batyrev94}.  Such a desingularization exists.
Pulling the origin and then the interior points of these faces gives a
regular subdivision of the lattice points of $\Pi$ \cite[Lemma~4.3.12]{DLRS10}
whose restriction to each such face is the star subdivision; it need not use
the lattice points interior to facets.  Place every lattice point on the lower
envelope of a height function inducing this subdivision and perturb by
$\varepsilon$ times a strictly convex function and then generically: for small
$\varepsilon$ this gives a regular triangulation that refines the subdivision
and uses every lattice point \cite[Lemma~2.3.16]{DLRS10}.  Every maximal simplex
contains the origin, the only interior lattice point of $\Pi$, and the
triangles $(c,v_j,v_{j+1})$, with $c$ the interior point of the face, are
unimodular and are not subdivided, so the triangulation still restricts to the
star subdivision on each such face.  The
resolution is then small over the nodes and is the blowup of the vertex at
each cone point, as Theorem~\ref{thm:necessity} requires.  An arbitrary
maximal triangulation need not have this property: it may cut off a
unimodular corner of a pentagon or hexagon (Section~\ref{sec:boundary}).  The polytope $\Delta^{\vee}$ of Section~\ref{sec:X0} is
admissible, its $68$ two-faces being $38$ unimodular triangles, $26$ unit
squares, two pentagons and two hexagons; $\Delta_{19}$ and $\Delta_{20}$
of Section~\ref{sec:test} are admissible; $\Delta$ itself is not, because
the quadrilateral face carrying its $\mathbb F_1$-cone point is neither a
square nor a triangle, but it satisfies the weaker hypothesis of
Proposition~\ref{prop:kernelequal} below.

\subsection{Local ray relations and \texorpdfstring{$T^1$}{T1}}
\label{sec:channelsT1}

\begin{proposition}\label{prop:T1channels}
Let $G$ be a singular $2$-face of $\Pi$ with unit edges and vertices
$v_1,\dots,v_n$ in cyclic order, and put
\[
R_G:=\Bigl\{a\in\QQ^{\,n}\ :\ \textstyle\sum_j a_jv_j=0\Bigr\}.
\]
Then $R_G$ has dimension $n-3$, Altmann's model gives a natural isomorphism
\[
T^1_{(Y_G,0)}\ \cong\ R_G\otimes_{\QQ}\CC.
\]
If $G$ is a unit square, or a pentagon or hexagon with one interior
lattice point, the rows of Section~\ref{sec:matrix} at $G$, read in the
coordinates of its rays, form a $\QQ$-basis of $R_G$.
In the pentagon and hexagon cases, the crepant resolution has an
exceptional surface $E$ over the germ, and $R_G$ is
$K^{\perp}_{\QQ}\subset\Pic(E)^{\vee}_{\QQ}$; when $G$ is a unit square it is
the line spanned by the intersection vector of the exceptional curve of a
small resolution.
\end{proposition}

\begin{proof}
Altmann shows that $T^1$ of the germ is concentrated in the Gorenstein
degree \cite[(6.5) and the subsequent Remark]{Altmann95} and identifies it
with the complex vector space of dilations of the edges of $G$ satisfying the
closing condition, modulo the homothety \cite[(2.2), (2.6)]{Altmann97}; its
dimension is $(n-2)-1$.  Send a dilation $t$, with $t_j$ the dilation of the edge
$(v_j,v_{j+1})$, to its jump $a_j:=t_j-t_{j-1}$ across the vertex $v_j$.
Then
\[
\sum_j a_jv_j=\sum_j t_jv_j-\sum_j t_jv_{j+1}
=-\sum_j t_j\,(v_{j+1}-v_j)=0
\]
by the closing condition; the kernel of $t\mapsto a$ is the constants.
The construction is defined over $\QQ$, and its rational form maps
isomorphically to $R_G$: both spaces have dimension $n-3$ because
$\Cone(G)$ is three-dimensional.  Extending scalars gives the stated
isomorphism with $T^1$.  For the last two clauses, note first that $G$ lies in
the affine hyperplane $\langle u_F,\cdot\rangle=-1$ for either facet
$F\supset G$, so $\sum_ja_jv_j=0$ forces $\sum_ja_j=0$.  If $G$ is a
pentagon or hexagon, the $v_j$ are the rays of the fan of $E$ seen from
the interior lattice point, and dualizing
$0\to M_E\to\QQ^{\,n}\to\Pic(E)_{\QQ}\to0$ identifies
$\{a:\sum_ja_jv_j=0\text{ in the plane of }G\}$ with
$\Pic(E)^{\vee}_{\QQ}$; since $K_E=-\sum_jD_j$ the extra condition
$\sum_ja_j=0$ is orthogonality to $K_E$, and a divisorial row of
Section~\ref{sec:matrix} is $\varphi\mapsto(\langle\varphi,D_j\rangle)_j$
for $\varphi$ in a basis of $K^{\perp}$, which is that identification.
If $G$ is a unit square with cyclic vertices $(a,b,c,d)$, the space is
one-dimensional and spanned by $(-1,1,-1,1)$, because
$v_a+v_c=v_b+v_d$; that is the node row of Section~\ref{sec:matrix}, the
intersection functional of $\Gamma$.
\end{proof}

Thus the $K^{\perp}$ lattice introduced in Section~\ref{sec:matrix} as a
substitute for the missing exceptional curve class is a natural rational
form of the tangent space to the germ's deformation functor.  A
linear relation among the rays of a cone represents a
curve class, with $a_j$ the intersection number against the invariant
divisor $D_j$; at a node the square's relation $v_a+v_c=v_b+v_d$ is the
class of the exceptional curve of a small resolution, and the row
$(-1,1,-1,1)$ of Section~\ref{sec:matrix} is its vector of intersection
numbers.  The proposition says that this reading survives at a del Pezzo
cone point, where the exceptional locus is a surface.  Two consequences
follow.  The interior lattice point of a pentagon or hexagon
carries coefficient $\langle\varphi,K\rangle=0$ in every row, so the rows
are supported on the vertices of the face; and the total number of
relation rows equals $\dim H^0(\mathcal T^1_X)$, which is $36$ for $X^{\circ}$
and $16$ for $X_{19}$.

\subsection{The ambient kernel is the topological kernel}
\label{sec:kernelequal}

\begin{proposition}\label{prop:kernelequal}
Suppose every $2$-face of $\Pi$ is either a unimodular triangle or has
dual edge of lattice length one.  Let $\Xh\to X$ be a smooth hypersurface
resolution induced by a maximal projective crepant partial
desingularization of $\PP_{\Pi}$.  Write $\ell^{*}(G)$ for the number of lattice
points in the relative interior of $G$, and $\ell(G^{*})$ for the lattice
length of the dual edge.  Then the correction term
$\sum_{G}\ell^{*}(G)\bigl(\ell(G^{*})-1\bigr)$, summed over $2$-faces $G$
of $\Pi$, in Batyrev's formula \cite[Corollary~4.5.1]{Batyrev94} for
$h^{1,1}$ vanishes, and the ambient divisor classes span
$H^2(\Xh;\QQ)=H^{1,1}(\Xh;\QQ)$.

Suppose in addition that $X$ has only isolated singularities, all nodes
or anticanonical cones over $\dP_6$, $\dP_7$ or $\mathbb F_1$, and that
$\Xh\to X$ is small over the nodes and is the blowup of the vertex at each
cone point.  Define $B$ by the pairings of
bases of $H_2(L_p;\QQ)$ with ambient divisor classes, as in
Sections~\ref{sec:matrix} and~\ref{sec:partner}.  Then
\[
\ker(B^{\mathsf T})=
\ker\Bigl(\bigoplus_pH_2(L_p;\QQ)\longrightarrow H_2(\Xh;\QQ)\Bigr).
\]
These additional hypotheses hold for admissible polytopes with the
resolution of Section~\ref{sec:combsetting} and for
$\Delta$ with the resolution of Section~\ref{sec:partner}.
\end{proposition}

\begin{proof}
A unimodular triangle has $\ell^{*}(G)=0$; every other $2$-face has
$\ell(G^{*})-1=0$ by hypothesis.  Hence every summand of the correction
term vanishes.  In the proof of \cite[Theorem~4.4.2]{Batyrev94}, the
divisor $D_{\rho}$ indexed by an interior lattice point of $G$ meets
$\Xh$ in $\ell(G^{*})$ irreducible components.  Thus
$\ell^{*}(G)\bigl(\ell(G^{*})-1\bigr)$ counts the additional components
beyond one for each such ambient divisor.  When the correction vanishes,
every nonempty restriction of a torus-invariant divisor is irreducible.
These restrictions span $H^2(\Xh;\QQ)$.  Indeed, let $Z\subset\Xh$ be the
part lying over the open torus $T\cong(\CC^*)^4$, a nondegenerate affine
hypersurface, and let $D=\Xh\smallsetminus Z$, a divisor with simple normal
crossings whose components are those restrictions.  By the Lefschetz
theorem of Danilov--Khovanskii \cite[Proposition~3.9 and Remark~3.10]{DK86} the restriction
$H^2(T;\QQ)\to H^2(Z;\QQ)$ is an isomorphism, so $H^2(Z;\QQ)$ is pure of
weight~$4$.  The image of the pure weight-$2$ space $H^2(\Xh;\QQ)$ in it is
therefore zero, so by the exact sequence
$H^2_D(\Xh;\QQ)\to H^2(\Xh;\QQ)\to H^2(Z;\QQ)$, in which
$H^2_D(\Xh;\QQ)\cong H_4(D;\QQ)$ is spanned by the components of $D$, the
space $H^2(\Xh;\QQ)$ is spanned by their classes.
The smooth Calabi--Yau hypersurface $\Xh$ has $h^{2,0}(\Xh)=0$, so
$H^2(\Xh;\QQ)=H^{1,1}(\Xh;\QQ)$.

Under the additional isolated-germ hypotheses, a row of $B$ records the
pairings of a class with the ambient divisors, so a combination of rows
vanishes exactly when the corresponding total class pairs to zero against
every ambient divisor.  The evaluation pairing
$H_2\times H^2\to\QQ$ is perfect; the ambient divisor classes span
$H^2$ by the first part, and therefore the total class pairs to zero
against all of them exactly when it vanishes in $H_2(\Xh;\QQ)$.
\end{proof}

The correction-term hypothesis alone does not imply that $X$ has isolated
singularities; the additional assumptions are needed to interpret the
matrix as the isolated-link pairing used in this paper.
Section~\ref{sec:kernelcomp} uses only the containment: a coordinate
vanishing on the ambient kernel also vanishes on the topological kernel.
The proposition gives equality for the admissible class and the treated
mirror partner.  This equality is needed when applying the converse
criterion of the companion paper~\cite{PaperV} to the matrix kernel.
The correction-term vanishing is distinct from the one that
empties Mavlyutov's mechanism of non-polynomial deformations
\cite{Mav04}: that is the mirror correction, summed over $2$-faces of
$\Pi^{*}$, and it follows from the unit edges of $\Pi$.  Both vanishings
hold for the admissible polytopes and for $\Delta$.

\subsection{\texorpdfstring{$\QQ$}{Q}-factoriality and primitive contractions}
\label{sec:qfact}

Gross proves that a primitive contraction of a divisor
to a point on a smooth Calabi--Yau threefold has smoothable target unless
the exceptional divisor is
$\PP^2$ or $\mathbb F_1$ \cite[Theorem~5.8]{Gross97}; degrees $6$ and
$7$ are allowed.  Namikawa proves that a $\QQ$-factorial
Calabi--Yau threefold with isolated toric Gorenstein singularities, possibly
several, has surjective map $\Def(X)\to\prod_p\Def(X,p)$, and hence is smoothable as soon
as its germs are \cite[Corollary~6]{Namikawa97}.  When a smooth Calabi--Yau
threefold contains a del Pezzo surface $E$ of degree~$6$ whose contraction
alone produces a single singular point $p$, he shows that the image of the
restriction of divisors to $E$ determines the image of the deformations of
the contracted threefold in the deformation space of its germ at $p$, and in
particular which smoothing components of the germ that image contains
\cite[Theorem~10]{Namikawa97}; the divisorial rows of
Section~\ref{sec:matrix} record that restriction modulo $K_E$.  For a threefold with
a single degree-$7$ cone point,
Theorem~\ref{thm:necessity} obstructs smoothing when the root line maps
injectively to $H_2(\Xh;\QQ)$, as stated in
Corollary~\ref{cor:onegerm}.  Gross's primitive contractions have
$\QQ$-factorial targets, for which the root line instead maps to zero.
The following proposition explains the distinction.

\begin{proposition}\label{prop:qfact}
Let $X$ be $\QQ$-factorial and let $\pi\colon\Xh\to X$ be a projective
crepant resolution.  Suppose every irreducible component $E_i$ of
$\operatorname{Exc}(\pi)$ is contracted to a point and the points $\pi(E_i)$
are pairwise distinct.  Then the rows of Section~\ref{sec:matrix} at each of
the corresponding divisorial germs vanish identically.
\end{proposition}

\begin{proof}
Because $X$ is $\QQ$-factorial and $\pi$ is projective and birational, the
exceptional locus of $\pi$ has pure codimension one
\cite[Cor.~2.63]{KollarMori}.  Thus its irreducible components
$E_1,\dots,E_r$ are divisors.  By hypothesis they lie over distinct points,
as they do for the divisorial germs of this paper.

$\QQ$-factoriality also makes $\pi^{*}$ defined on $\Cl(X)_{\QQ}$, and
then $\Cl(\Xh)_{\QQ}=\pi^{*}\Cl(X)_{\QQ}\oplus\bigoplus_i\QQ E_i$, with
$\Cl(\Xh)=\Pic(\Xh)$ since $\Xh$ is smooth.  Fix $i$ and write $E=E_i$.
A pullback $\pi^{*}D$ restricts trivially to $E$ because $\pi(E)$ is a
point, and $E_j|_E=0$ for $j\neq i$ because $E_j$ and $E$ lie over
distinct points, so the image of $\Pic(\Xh)_\QQ\to\Pic(E)_{\QQ}$ is
spanned by $E|_E$.  That class is $K_E$: adjunction gives
$K_E=(K_{\Xh}+E)|_E$, and $K_{\Xh}|_E=(\pi^{*}K_X)|_E=0$ by crepancy and
the same observation.  Every divisor class therefore pairs to zero
against $K_E^{\perp}$, and the rows, which are those pairings, vanish.
\end{proof}

Under the hypotheses of Proposition~\ref{prop:qfact}, the rows at every
divisorial germ are zero, so the kernel contains its entire
$K^{\perp}$ and the test imposes no obstruction there.  In the primitive
type~II case with exceptional divisor $\dP_6$ or $\dP_7$,
\cite[Theorem~5.8]{Gross97} supplies the smoothing.
The proposition also
says that a $\QQ$-factorial $X$ admits no \emph{projective} crepant
resolution with an exceptional curve; it does not say that $X$ has no
nodes, and indeed a general one-nodal quintic threefold is factorial.
Contrapositively:

\begin{corollary}\label{cor:defect}
If the block of rows at a divisorial germ has positive rank, then $X$ is not
$\QQ$-factorial.  Since primitivity in \cite[Theorem~5.8]{Gross97}
supplies $\QQ$-factoriality of the target, that theorem, like
\cite[Corollary~6]{Namikawa97}, is silent on
every $X$ carrying such a germ, and the statements never
disagree. \qed
\end{corollary}

Here $\QQ$-factoriality is algebraic: every algebraic Weil divisor on $X$
is $\QQ$-Cartier, or equivalently
$\bigl(\Cl(X)/\Pic(X)\bigr)\otimes\QQ=0$.  This must be distinguished
from $\QQ$-factoriality of the analytic germ.  The analytic cone germs
here are not $\QQ$-factorial: at a cone over a del Pezzo surface $E$
the analytic local class group is $\Pic(E)/\ZZ K_E$, of rank
$\rho(E)-1$, which is Namikawa's invariant $\sigma$ \cite{Namikawa97}.  For
the toric del Pezzos of the framework that rank is also $\dim T^1=n-3$
(Section~\ref{sec:channelsT1}).  In general Namikawa proves
$\mu+\sigma\le\dim T^1\le2\mu+\sigma$ for isolated rational Gorenstein
threefold singularities \cite[Theorem~1]{Namikawa97}, where $\mu$ is the
dimension of the cokernel of
$d\log\colon H^1(\widetilde Y,\cO^*)\otimes\CC\to H^1(\widetilde Y,\Omega^1)$
on a good resolution $\widetilde Y$ (it is not the Milnor number), so
$\sigma$ and $\dim T^1$ can differ only when $\mu>0$.
A threefold can be algebraically $\QQ$-factorial while its analytic germ
is not: the target of a primitive type~II contraction
is the example that matters here, and a general one-nodal quintic is the
familiar one.

Thus a positive block rank witnesses the failure of $\QQ$-factoriality,
whereas the relation kernel measures cancellation
among the local classes.  The two quantities are not interchangeable: on
$X_7$ the block rank is $2$ while the relation kernel is zero.  On $X^{\circ}$
every block has the maximal rank, $2$ at each $\dP_7$ point and $3$ at each
$\dP_6$ point.

For admissible polytopes every divisorial block has positive rank.
Indeed, the exceptional divisor $E$
over a singular $2$-face is a toric surface whose boundary divisors are the
restrictions of the ambient toric divisors, so
$\Pic(\PP_{\widehat\Sigma})_\QQ\to\Pic(E)_\QQ$ is surjective and the block
at that germ has the full rank $n-3$ of
Proposition~\ref{prop:T1channels}, which is $2$ at a pentagon and $3$ at a
hexagon.  Thus no admissible $X$ carrying a pentagonal or hexagonal face is
$\QQ$-factorial; a node contributes no exceptional divisor to which this
argument could be applied.
Proposition~\ref{prop:qfact} therefore functions here as a consistency
check against \cite[Theorem~5.8]{Gross97} rather than as a discriminating
invariant.  In the purely nodal case, a $\QQ$-factorial threefold has no
small projective resolution.  On a small analytic resolution of a general
one-nodal quintic, the exceptional curve class vanishes in $H_2$, so the
entire local line belongs to the relation kernel and the necessity test
does not obstruct smoothing (Section~\ref{sec:test}).  The mixed
criterion uses $K^{\perp}$ in place of this local curve class.

\section{The criterion as a finite test}
\label{sec:test}

\subsection{The test}\label{sec:thetest}
Theorem~\ref{thm:necessity} is a finite test for any concrete
candidate: compute the matrix of Section~\ref{sec:matrix} from the
polytope (node rows; $K^{\perp}$ bases of the star fans of
the pentagon and hexagon faces), restrict to the branch subspaces,
and check whether a coordinate required to be nonzero by
Theorem~\ref{thm:necessity} is forced to vanish.  A
forced node coordinate contradicts Lemma~\ref{lem:P}; a forced
$\dP_7$ root coordinate contradicts Lemma~\ref{lem:D}; a forced
$s_3$-coordinate contradicts Lemma~\ref{lem:D} on profiles through
the line component.  The test uses nothing about the example beyond Batyrev regularity and its face inventory.

Applied across the admissible polytopes of
Definition~\ref{def:admissible} the test has a wide range, and the two
ends of that range are worth separating.  At the small end it is nearly immediate: with a single germ the criterion
asks only whether the germ's root subspace survives in $H_2(\Xh;\QQ)$, which
is a two-row rank computation.  That case is
Section~\ref{sec:onegerm}, and Section~\ref{sec:census} shows it is not
isolated: among admissible polytopes with at most nine vertices
carrying a pentagonal $2$-face, seventy-six of seventy-seven are
obstructed, sixty-seven of them with one singular point.  The interest
of $X^{\circ}$ is that it is at the other end: thirty germs of three
kinds, both $\dP_6$ branch profiles, and a $36\times26$ matrix of rank
$21$ whose $15$-dimensional kernel forces nothing at a node, and where the
four coordinates that decide the question appear once the $\dP_7$ branch
restriction is imposed.  Sections~\ref{sec:X20} and~\ref{sec:X19}
sit between the two, one where the obstruction is carried entirely by
the divisorial direction and one where the test is silent.

\subsection{A one-germ illustration}\label{sec:onegerm}
With one singular point, the criterion reduces to whether the local root line
survives in the homology of the crepant resolution.

\begin{corollary}\label{cor:onegerm}
Let $X$ be as in Theorem~\ref{thm:necessity} with exactly one singular
point $p$, an anticanonical $\dP_7$ cone, and suppose the $(-2)$-line
$\Rsm_p\subset H_2(L_p;\QQ)$ maps injectively into $H_2(\Xh;\QQ)$.  Then
$X$ admits no one-parameter smoothing.
\end{corollary}

\begin{proof}
With one singular point $K(X)=\{\kappa\in \Rsm_p:\kappa\mapsto0$ in
$H_2(\Xh;\QQ)\}$, which is $0$ by hypothesis.  A one-parameter smoothing
would produce, by Theorem~\ref{thm:necessity}(2), a
$\kappa\in K(X)$ whose $p$-component is a nonzero element of $\Rsm_p$.
\end{proof}

The hypothesis is a finite check: it suffices that the $(-2)$-row of
Section~\ref{sec:matrix} be nonzero, since a class pairing nontrivially
with an ambient divisor does not die in $H_2(\Xh;\QQ)$.

Such an $X$ exists at the seven-vertex bottom of the pentagonal census.  Let
\[
\begin{gathered}
\Delta_7=\conv\{(1,0,0,0),(0,1,0,0),(0,0,1,0),(-6,-4,-1,0),\\
(0,0,0,1),(-6,-4,0,-1),(-3,-2,1,-1)\},
\end{gathered}
\]
reflexive with seven vertices, seven facets and sixteen edges, all of
lattice length one.  Of its sixteen $2$-faces, fifteen are unimodular
triangles and one is a pentagon with a single interior lattice point and
dual edge of lattice length one, so $\Delta_7$ is admissible and
$X_7\subset\PP_{\Delta_7}$ has \emph{exactly one} singular point, an
anticanonical $\dP_7$ cone, which is smoothable as a germ.  As in
Section~\ref{sec:X0}, $\omega_{X_7}\cong\cO_{X_7}$ and
$h^1(\cO_{X_7})=0$, and $\Xh_7$ has $h^{1,1}=4$ and $h^{2,1}=214$.  The
two divisorial rows are linearly independent, so the relation-row block has
rank $2$, the maximum, the ambient kernel is zero, and
Corollary~\ref{cor:onegerm} applies: \emph{$X_7$ admits no one-parameter
smoothing.}

Since the block has positive rank,
Corollary~\ref{cor:defect} says $X_7$ is not $\QQ$-factorial, so
\cite[Theorem~5.8]{Gross97} does not apply and there is
no conflict with the smoothability of primitive type~II contractions
with $\dP_7$ exceptional divisor.  For $X_7$, the obstruction requires only
the rank calculation for its two divisorial rows.  At first order the same
data show that every first-order deformation of $X_7$ is locally trivial at
its singular point: by \cite{PaperV} (see Section~\ref{sec:discussion}) the
image of the localization map on first-order deformations is the
complexified relation kernel, which is zero here; compare
\cite[Theorem~10]{Namikawa97} for a contracted del Pezzo surface of
degree~$6$.  Vanishing at first order does not by itself exclude a
smoothing whose local classifying arc vanishes to order two or more; the
period argument behind Corollary~\ref{cor:onegerm} applies to arcs of every
order and excludes it.

\subsection{A census of the admissible polytopes with a pentagonal
\texorpdfstring{$2$}{2}-face, at up to nine vertices}
\label{sec:census}
As a bounded computational application, the test can be applied to every
polytope in a finite part of the Kreuzer--Skarke classification.

\begin{proposition}\label{prop:census}
Among the $3\,905\,789$ reflexive $4$-polytopes with at most nine vertices
in the Kreuzer--Skarke classification \cite{KS00}, exactly $77$ are
admissible and carry
a pentagonal $2$-face, one pentagon apiece: one at seven vertices, thirteen
at eight, sixty-three at nine.  No pentagonal $2$-face occurs below seven
vertices.  None of the $77$ carries a hexagonal $2$-face, so each of the
corresponding $X$ has exactly one $\dP_7$ cone point and no $\dP_6$ point;
the search records that each has at most two nodes, the inventories being
$67$ with the cone point alone, $9$ with one node in addition, and $1$ with
two nodes.  The test of Theorem~\ref{thm:necessity} obstructs all but one of
them, so those $76$ admit no one-parameter smoothing, and by
Corollary~\ref{cor:def} they are not smoothable over any base germ.  On the remaining polytope, which
has nine vertices and two nodes, the test is silent.
\end{proposition}

\begin{proof}
Each such $X$ satisfies the standing hypotheses of
Theorem~\ref{thm:necessity} by Section~\ref{sec:combsetting}, with the
crepant resolution chosen there: the restriction of a maximal projective
crepant partial desingularization of $\PP_{\Pi}$ that refines the star
subdivision at the interior point of the pentagon, small
over the nodes, and the blowup of the vertex, with exceptional surface
$\dP_7$, at the cone point.
The kernel calculations of Section~\ref{sec:thetest} are
recomputed for each of the $77$ from its vertex list
(Appendix~\ref{app:cert}).  Since no $\dP_6$ point occurs, the only
branch restriction is $\alpha=0$ at the pentagon, and a forced node or
$(-2)$-coordinate is contradicted by Lemma~\ref{lem:P} or
Lemma~\ref{lem:D}.
\end{proof}

The enumeration itself is the admissibility-restricted extract of a
scan reported in the companion paper~\cite{PaperV}.  The input files are
the at-most-nine-vertex part of the complete database copy verified in
\cite[\S6.1]{PaperI}: every entry is reflexive and the normal forms are
pairwise distinct up to lattice equivalence.  The displayed input count
is read from these files.  The germ inventory and
the matrix are recomputed here from each stored vertex list.  The
$67$ single-germ cases are instances of Corollary~\ref{cor:onegerm}.
For these cases and the $9$ cases with one node in addition, the full
matrix has zero relation kernel, so the required relation cannot exist.
Together with the distinguished example of Section~\ref{sec:proofs}, the
census therefore gives seventy-six further nonsmoothable threefolds, none of them
$X^{\circ}$, $X_{20}$ or $X_{19}$, whose polytopes have $26$, $20$
and $19$ vertices and lie outside it.

The enumeration also identifies the single polytope among the $77$ for
which the obstruction vanishes.  The companion paper~\cite{PaperV} proves
that its hypersurface is smoothable.

The census is restricted to polytopes with a pentagonal $2$-face, and that
restriction is not vacuous: admissible polytopes in the same range carry
hexagonal $2$-faces, hence $\dP_6$ cone points.  For example, an eight-vertex
case is obtained from $\Delta_7$ by adjoining $(-3,-2,-1,1)$.  On
such an $X$ Theorem~\ref{thm:necessity} constrains the one-dimensional $\dP_6$
branch and says nothing about the $A_2$ branch, so the conclusions above do not
extend to them and we do not claim them.

\subsection{An obstruction the nodal theory cannot see}\label{sec:X20}
The test also obstructs threefolds that no purely nodal argument can
reach.  Let $\Delta_{20}$ be the $20$-vertex reflexive polytope
\[
\begin{gathered}
\conv\{(1,0,0,0),(0,1,0,0),(0,0,1,0),(0,0,0,1),(1,-1,-1,-1),
(-1,1,1,1),(0,0,0,-1),\\
(0,0,-1,0),(0,-1,0,0),(-1,1,0,1),(0,0,-1,-1),(-1,1,0,0),
(-1,0,1,1),(0,-1,0,-1),\\
(-1,0,1,0),(0,-1,-1,-1),(-1,0,0,1),(-1,0,-1,0),(-1,-1,0,0),
(-1,-1,1,0)\},
\end{gathered}
\]
which appears in the accompanying data of \cite{PaperIII}.  Its generic
anticanonical hypersurface $X_{20}$ has seventeen ordinary double
points and one $\dP_7$-cone point, all isolated; as in
Section~\ref{sec:X0} it satisfies
$\omega_{X_{20}}\cong\cO_{X_{20}}$ and $h^1(\cO_{X_{20}})=0$, and a
smooth projective crepant subdivision exists, obtained as in
Section~\ref{sec:matrix} from the seventeen square diagonals and the
star subdivision at the interior point of the pentagon.  Here
$\Delta_{20}$ has $22$ lattice points and no facet with an interior
point, so Batyrev's formulas give
$h^{1,1}(\Xh_{20})=22-5=17=\operatorname{rank}\Pic$ of the ambient
toric fourfold: the ambient divisor classes span
$H^2(\Xh_{20};\QQ)$, and the ambient kernel of
Section~\ref{sec:kernelcomp} coincides with the topological one rather
than merely containing it, as Proposition~\ref{prop:kernelequal} also
gives.  The outcome is:
\begin{enumerate}
\item the nodal subsystem has rank $11$ of $17$ and \emph{no coloop},
so, the two kernels agreeing, every exceptional curve of a small
resolution does appear with nonzero coefficient in some relation among
them: Friedman's necessary condition is satisfied;
\item the mixed matrix has $19$ rows on $21$ rays, its kernel has
dimension $6$, and \emph{both} $\dP_7$ coordinates vanish identically
on it, the root coordinate among them.
\end{enumerate}
By Theorem~\ref{thm:necessity}(2) and Lemma~\ref{lem:D}, $X_{20}$
admits no one-parameter smoothing.  Here the obstruction is carried
entirely by the divisorial direction, and a reader applying
\cite{Friedman86,RT09} alone would find nothing wrong.  On
$X^{\circ}$, by contrast, the two nodal coloops already signal that
the nodal theory is insufficient.

\subsection{Where the test is silent}\label{sec:X19}
We also ran the test on the anticanonical hypersurface $X_{19}$ of the
$19$-vertex reflexive polytope $\Delta_{19}$, whose generic member has fourteen
nodes and one $\dP_7$-cone point, all isolated.  The present necessity test is silent on this example, which the
companion paper~\cite{PaperV} proves smoothable.  The vertices of
$\Delta_{19}\subset N\cong\ZZ^4$ are
\[
\begin{array}{llll}
(1,0,0,0), & (0,1,0,0), & (0,0,1,0), & (0,0,-1,0),\\
(0,-1,0,0), & (1,-1,-1,0), & (0,0,0,1), & (-1,1,1,-1),\\
(0,-1,0,1), & (0,0,-1,1), & (-1,1,0,-1), & (-1,0,1,-1),\\
(-1,0,1,0), & (-1,1,0,0), & (0,-1,-1,0), & (0,-1,-1,1),\\
(-1,0,0,-1), & (-2,1,1,-1), & (-1,0,0,1). &
\end{array}
\]
It is the polytope of \cite[Section~5.3]{PaperIII}.  The outcome is
structurally opposite to $X^{\circ}$:
\begin{enumerate}
\item the node subsystem has rank $12$ with two coloops: as for
$X^{\circ}$, no purely nodal mechanism can smooth $X_{19}$;
\item on the full kernel (dimension $3$) the complementary
$\dP_7$-coordinate $\alpha$ is already forced to zero, so every
relation class lies on the $(-2)$-line without imposing the branch
restriction; and
\item no coordinate required by Theorem~\ref{thm:necessity} is forced to
vanish: a generic kernel class is
nonzero in all fourteen node coordinates and in the $(-2)$-coordinate
simultaneously.
\end{enumerate}
The necessity theory of this paper is therefore silent on $X_{19}$,
and the companion paper~\cite{PaperV} proves the required mixed
smoothing exists.  As a control in the opposite
direction, the test applied to the generic one-nodal quintic
threefold is silent as well, as it must be: there the exceptional
curve class vanishes in $H_2$ of the small resolution (the threefold
is $\QQ$-factorial), the kernel is the whole line $H_2(L;\QQ)$, and
the node coordinate is realized, the smoothable side of the nodal
picture.  Every smoothing of $X_{19}$ must use the $\dP_7$ direction.
The companion paper~\cite{PaperV} proves its existence by the mixed
deformation criterion and also studies the limitations of deformations
induced by a toric ambient space.

\section{Discussion}\label{sec:discussion}

\subsection{The two-dimensional degree-six smoothing component}\label{sec:open}
Theorem~\ref{thm:necessity} treats $\dP_6$ points smoothed through
the one-dimensional component.  The companion paper~\cite{PaperV}
treats the two-dimensional component by an explicit simultaneous partial
resolution that replaces the central cone by three nodes and is an
isomorphism on every noncentral fiber.  In the marking
\[
R_{A_2}=\left\{\mu_1e_1+\mu_2e_2+\mu_3e_3:
\mu_1+\mu_2+\mu_3=0\right\},
\]
the required condition is $\mu_1\mu_2\mu_3\ne0$.
It is stronger than requiring a nonzero vector in $R_{A_2}$ and holds
for arbitrary smoothing arcs.  None of the proofs in the present
paper require that extension.

\subsection{First-order theory}
The companion paper~\cite{PaperV} identifies the image of the
localization map on first-order deformations with the topological
relation kernel, using the local-to-global comparison of
Friedman--Laza~\cite{FL22a}.  Under
Proposition~\ref{prop:kernelequal}, this is the matrix kernel used here.
The period argument of the present paper gives necessity directly for
an actual smoothing arc and does not assume that its first derivative
smooths each germ.  In the purely nodal case, the criterion for actual
smoothings is already \cite[Thm.~1.3]{RT09}.

\subsection{Sufficiency}\label{sec:sufficiency}
The companion paper~\cite[Thms.~A--B]{PaperV} proves an exact
converse for connected normal projective threefolds with trivial
dualizing sheaf, $H^1(X,\cO_X)=0$, and the prescribed analytic germs:
nodes and anticanonical cones over smooth del Pezzo surfaces of degrees
five, six and seven, or over $\PP^1\times\PP^1$.
For each choice of local components, the relation kernel must avoid
every local discriminant: a nonzero coefficient in each one-dimensional
local summand, three nonzero coefficients in each $A_2$ summand, and five
nonzero conic-pencil pairings at each degree-five point.
Thus silence of the test in the present paper gives the converse only
when it includes all the nonvanishing conditions for the chosen
components.

The proof replaces the degree-six and degree-seven cones by proper
crepant partial models whose full local deformation bases are smooth.
Friedman's generalized $T^1$ lifting theorem, together with marked
local lifting, makes the partial model unobstructed; analytic
blow-down then gives smoothness of the selected base of $X$ whenever
its relation kernel projects nontrivially to every degree-six
summand.  No $\QQ$-factoriality assumption is imposed on $X$, and the
partial model need not be projective.  The criterion also proves the
smoothability stated in Example~\ref{q:X19}.  The separate issue of which
deformations can be induced by a toric ambient family is studied in
\cite{PaperV}.

\subsection{Mirror symmetry}\label{sec:mirror}
For the admissible hypersurfaces of Definition~\ref{def:admissible}, the
companion paper~\cite{PaperVI} studies the relation lattice on the mirror
side.  For each profile in the sense of Section~\ref{sec:matrix} it passes to a
nodal crepant partial resolution induced by a projective toric morphism.  It
identifies the integral relations among the exceptional curves of a small
resolution with the relations among the vanishing three-spheres of suitable
degenerations of the Batyrev mirror family, and it characterizes
smoothability by relations among those spheres.

\appendix

\section{Computational certificates}\label{app:cert}

The exact lattice and matrix claims used in the proofs are reproduced by the
Python programs listed below, in exact rational arithmetic and independently of
any computer-algebra system.  A third self-contained
program combines exact integral checks with a floating-point diagnostic
of the period parametrization, whose underlying identity is proved
exactly in Lemma~\ref{lem:D1}.  The inputs are
Tables~\ref{tab:rays} and~\ref{tab:kperp}, the ancillary subdivision
and restriction data documents, the polytope of
\cite[Section~5.3]{PaperIII}, and the vertex lists of the $77$ polytopes
of Proposition~\ref{prop:census}:
\begin{itemize}
\item the lattice package (\texttt{milnor\_kernels.py}): the boundary structure of the exceptional
surfaces, the norms and pairings of the $K^{\perp}$ rows, the root
decompositions, the $(-2)$-enumerations with rigorous search bounds,
and the isometries to the standard del Pezzo markings;
\item the matrix package (\texttt{global\_kernel.py}): the reconstruction of all $36$ rows from
the printed data, the ranks, coloops, kernels, the four profiles and
the enlarged profile with their forced-zero lists, and the
criticality statement of Section~\ref{sec:kernelcomp}(3);
\item the period package (\texttt{dp\_periods.py}): the exact fan combinatorics behind
\eqref{eq:omegaM} (the $\pm1$ pairings on all rays, the divisor
computation, the slice self-intersection patterns), the corner-chart
transitions, and the constancy of the pullback in
Lemma~\ref{lem:D1}, verified to machine precision at random points;
\item \texttt{examples.py}, the general implementation of
Section~\ref{sec:test}, which recomputes every number of
Section~\ref{sec:kernelcomp} for $X^{\circ}$ from the polytope,
and the $\Delta_{19}$ and $\Delta_{20}$ computations, including the
absence of nodal coloops on $\Delta_{20}$;
\item the one-germ package (\texttt{lone\_germ.py}): the reflexive
$4$-polytope $\Delta_7$ of Section~\ref{sec:onegerm}, its face
inventory and admissibility, the rank-$2$ relation-row block and zero
kernel, and the Hodge numbers $(4,214)$;
\item the census package (\texttt{classification.py}, reading the
extract \texttt{framework\_77.json}): the germ inventory, the matrix and
the verdict for each of the $77$ polytopes of
Proposition~\ref{prop:census}, rebuilt from the stored vertex lists;
\item the relation-row package (\texttt{relation\_class.py}): for
$X^{\circ}$, $X_{20}$, $X_{19}$ and $X_7$, that every relation row
satisfies $\sum_ja_jv_j=0$ and $\sum_ja_j=0$, carries coefficient zero
at the interior lattice point, and that the rows at a germ span the
whole ray-relation space of rank $n-3$
(Proposition~\ref{prop:T1channels}); together with the per-germ block
ranks of Corollary~\ref{cor:defect} and an independent reconstruction of
$\rk B=21$ for $X^{\circ}$;
\item the kernel-equality package (\texttt{kernel\_equality.py}): the
vanishing of both Batyrev correction terms, and the reason each $2$-face
contributes zero, for $\Delta^{\vee}$, $\Delta$, $\Delta_{20}$,
$\Delta_{19}$, $\Delta_7$ and all $77$ census polytopes
(Proposition~\ref{prop:kernelequal});
\item the mirror-partner package (\texttt{mirror\_partner.py}): the
reconstruction of the twenty-four rows of $B_{\Delta}$ from
Tables~\ref{tab:partnerrays} and~\ref{tab:partnersurf}, including the
derivation of both exceptional surface lattices from the face data of
Table~\ref{tab:partnerrays}, the reduced-rigidity of the
$\mathbb F_1$ germ by the zero-sum test on its edge vectors, the two
vanishing columns of item~(3), the root decomposition $A_1\oplus A_2$ of
$K^{\perp}(\dP_6)$, the absence of $(-2)$-classes in
$K^{\perp}(\mathbb F_1)$ and the square $-8$ of its generator, and the
ranks, coloops and forced-zero lists of Section~\ref{sec:partner}.
\end{itemize}
The fixed subdivision of Section~\ref{sec:matrix} and its
smoothness and projectivity are certified by an exact integral
support function; the Batyrev regularity of the reference equation
of $X^{\circ}$, including nondegeneracy along every face, is
certified at an integral witness (\texttt{resolution\_fan.sage} and
\texttt{global\_section.sage}, with five further Sage certificates
and the face-data module \texttt{fixed\_example.py}).  The
subdivision of Section~\ref{sec:partner}, its $94$ unimodular maximal
cones, the two stars and the twenty diagonals are certified the same
way by \texttt{mirror\_partner\_fan.sage}.

The arXiv ancillary bundle includes this directory together with the
shared exact toric modules used by the
programs.  The same files are public, together with the programs and
data of the companion papers, at
\begin{center}
\url{https://github.com/bwuebben/calabi-yau-smoothability}\,.
\end{center}

\bigskip
\noindent{\sc Bernd Johannes Wuebben, New York, NY,}
\texttt{wuebben@gmail.com}

\end{document}